\documentclass[type=article,thmsty=normal]{notes}
\usepackage{mathdef}

\usepackage{xcolor} 
\definecolor{brightcerulean}{rgb}{0.11, 0.67, 0.84}
\definecolor{bondiblue}{rgb}{0.0, 0.58, 0.71}
\definecolor{darkspringgreen}{rgb}{0.09, 0.45, 0.27}
\definecolor{darkred}{rgb}{0.55, 0.0, 0.0}
\colorlet{mdtRed}{red!50!black}
\hypersetup{pdfpagemode={UseOutlines},
bookmarksopen=true,
bookmarksopenlevel=0,
hypertexnames=false,
colorlinks=true,
citecolor=darkspringgreen,
linkcolor=darkred,
urlcolor=darkred,
pdfstartview={FitV},
unicode,
breaklinks=true,
}

\usepackage[style=numeric-comp,giveninits=true,backend=biber,natbib=true,maxbibnames=99,doi=false,isbn=false,url=false]{biblatex}
\title{High-Order Talagrand and Eldan--Gross Inequalities via Besov-Type Variance Functionals}

\author{\begin{tabular}{C{6cm}C{6cm}}
{Guangyue \textsc{Han}} & {Peijie \textsc{Li}}\\
\small The University of Hong Kong & \small The University of Hong Kong\\
\small \textit{\email{ghan@hku.hk}}  & \small \textit{\email{lipeijie98@connect.hku.hk}} \\
\end{tabular}
}

\date{}

\begin{document}

\maketitle

\begin{abstract}
By introducing high-order Besov-type variance functionals that generalize the canonical variance, we develop a unified framework for proving high-order Talagrand-type inequalities that relate high-order energies to Fourier weights. Applying this machinery, we establish high-order Poincar\'e-type, $\Leb^p$--$\Leb^q$, isoperimetric-type, Falik--Samorodnitsky and Eldan--Gross inequalities, all with explicit constants, in both the Boolean and Gaussian settings. Fundamentally, our semigroup-based framework relies primarily on hypercontractivity and high-order Bismut-type derivative estimates, and is broadly applicable.

\vspace{2mm}

\vspace{2mm}
\end{abstract}

\section{Introduction}
\label{Section:Introduction}

\subsection{Background}

Talagrand-type inequalities have a long history in Boolean function analysis, with deep connections to isoperimetry, sharp threshold phenomena, and concentration of measure in high-dimensional analysis. Initiated by Talagrand's seminal works~\cite{Talagrand1993Iso,Talagrand1994OnRA,Talagrand1997OnBA} and further developed in a long line of research culminating in the recent breakthrough of Eldan--Gross~\cite{Eldan2022Concentration}, these inequalities strengthen classical results such as the Poincar\'{e} inequality and the Kahn--Kalai--Linial (KKL) theorem~\cite{KKL1988}, providing logarithmically amplified dimension-free bounds relating variances of functions to their first-order energies.

More specifically, for functions on the $n$-dimensional Boolean cube $\cube^n$, let $\Var$ and $\left\|\,\cdot\,\right\|_{\Leb^p}$ denote the variance and $\Leb^p$ norm taken w.r.t. the uniform probability measure, and let $\left|\grad \cdot\,\right|^2=\sum_i \left|\Dop_i \cdot\,\right|^2$ denote the (squared) modulus of discrete gradient, where $\Dop_i$ is the discrete derivative in the $i$-th coordinate. The Poincar\'e inequality then provides an elementary bound between energy and variance: for all $f:\cube^n\to \Rbb$,
\begin{equation}\label{eq:boolean-Poincare}
    \left\|\left|\grad f\right|\right\|_{\Leb^2}^2 \geq \Var(f).
\end{equation}
In the famous paper by Kahn--Kalai--Linial~\cite{KKL1988}, a logarithmic strengthening of \eqref{eq:boolean-Poincare} was established for Boolean functions $f:\cube^n\to \cube$:
\begin{equation}\label{eq:KKL}
    \max_{i\in [n]}\Inf_i(f) \gtrsim \Var(f)\, \frac{\ln(n)}{n},
\end{equation}
where $\Inf_i(f):=\left\|\Dop_i f\right\|_{\Leb^2}^2$ is the influence (or ``partial energy'') of coordinate $i$ on $f$. Note that for $\cube$-valued $f$, $\left|\Dop_i f\right|$ is $\{0,1\}$-valued, thus $\Inf_i(f)=\left\|\Dop_i f\right\|_{\Leb^p}^p$ for all $p\geq 1$.

In \cite{Talagrand1993Iso}, Talagrand proved the following $\Leb^1$ isoperimetric-type inequality, which yields another logarithmic strengthening of \eqref{eq:boolean-Poincare}: for all $\func{f}{\cube^n}{\cube}$,
\begin{equation}\label{eq:Tal-iso1}
    \left\|\left|\grad f\right|\right\|_{\Leb^1} \gtrsim \Var(f) \sqrt{1+\ln \left(\frac{1}{\Var(f)}\right)}.
\end{equation}
Compared to the KKL inequality \eqref{eq:KKL}, which is tight for the Tribes functions but off by $\sqrt{n}/\log(n)$ for the Majority functions (see \cite[Sections 4 \& 5]{Ryan2014book}), Talagrand's isoperimetric-type inequality \eqref{eq:Tal-iso1} is tight for the Majority functions rather than the Tribes functions. Moreover, \eqref{eq:Tal-iso1} can also be seen as a quantitative analogue of the Gaussian isoperimetric inequality on the Boolean cube,  and is further strengthened by the Bobkov inequality in \cite{Bobkov1997isoperimetric}.

Talagrand also improved the KKL inequality~\eqref{eq:KKL} in \cite{Talagrand1994OnRA}, by showing an $\Leb^2$--$\Leb^1$ influence inequality: for all $\func{f}{\cube^{n}}{\Rbb}$,
\begin{equation}\label{eq:Tal-KKL}
    \sum_{i\in [n]} \frac{\left\|\Dop_i f\right\|_{\Leb^2}^2}{1+\ln(\left\|\Dop_i f\right\|_{\Leb^2}/\left\|\Dop_i f\right\|_{\Leb^1})} \gtrsim \Var(f).
\end{equation}
This inequality has since been extensively studied and extended to multifarious variants and settings: \cite{CL2012hypercontr} developed a semigroup approach extending \eqref{eq:Tal-KKL} to general hypercontractive settings and $\Leb^p$ variants; \cite{CE2023} studied \eqref{eq:Tal-KKL} in different types of Banach space codomain;
\cite{Rouze2024quantum,Blecher2024geometricinfluencesquantumboolean,JLZ2025} proved noncommutative analogues of \eqref{eq:Tal-KKL} in quantum settings; \cite{Tanguy2020,P2025} established high-order generalizations of \eqref{eq:Tal-KKL} in the Boolean and Gaussian settings, relating high-order influences (moments of high-order partial derivatives) and high-order Fourier weights, and further yielding high-order analogues of the KKL inequality \eqref{eq:KKL}; extending the method in \cite{P2025} to the quantum setting, \cite{Blecher2024geometricinfluencesquantumboolean} also proved a quantum $\Leb^1$ high-order variant of \eqref{eq:Tal-KKL}.

Recently in \cite{Eldan2022Concentration}, Eldan and Gross established the following breakthrough inequality: for all $\func{f}{\cube^n}{\cube}$ and $p\in [1,2]$,
\begin{equation}\label{eq:Eldan-Gross}
    \left\|\left|\grad f\right|\right\|_{\Leb^p}^p \gtrsim \Var(f) \left[\ln \left(2+\frac{\upe}{\sum_{i} \Inf_i(f)^2}\right)\right]^{\frac{p}{2}}.
\end{equation}
When $p=1$, the Eldan--Gross inequality \eqref{eq:Eldan-Gross} validates a conjecture of Talagrand \cite{Talagrand1997OnBA} that for monotone functions $\func{f}{\cube^n}{\cube}$ ($f(x)\leq f(y)$ whenever $x_i\leq y_i$ for all $i$), the logarithmic term $\ln(1/\Var(f))$ in \eqref{eq:Tal-iso1} can be improved to $\ln^{+}(1/\sum_i \Inf_i(f)^2)$.
When $p=2$, \eqref{eq:Eldan-Gross} improves the result of Falik--Samorodnitsky \cite{Falik-Samorodnitsky2007Edge-Isoperimetric} for $\cube$-valued functions and further implies the KKL inequality \eqref{eq:KKL}. Alternative proofs via combinatorial and semigroup approaches were later given in~\cite{rosenthal2020ramon,EKLM2025,IZ2026}, providing simpler arguments and sharper insights for the Eldan--Gross inequality \eqref{eq:Eldan-Gross}.

In a recent work \cite{ChangLi2026}, the authors introduced a unified semigroup-based \textit{variance-decay framework} for proving first-order Talagrand-type inequalities in the quantum Boolean setting. They established the principle that the short-time sharp decay of variance along the semigroup trajectories, which can be easily estimated via hypercontractivity, directly leads to a Talagrand-type inequality with the corresponding rate. Following this observation, they deduced quantum versions of classical results including Talagrand's isoperimetric-type inequality \eqref{eq:Tal-iso1} and the Eldan--Gross inequality \eqref{eq:Eldan-Gross}, together with an $\Leb^p$--$\Leb^q$ quantum energy-type analogue of \eqref{eq:Tal-KKL}. Applying this framework to high-order partial variance functionals (direct generalizations of the semigroup representation of variance), they also obtained a high-order quantum version of the influence inequality \eqref{eq:Tal-KKL}, albeit with a weaker logarithmic power ($1$ instead of the optimal $k$) compared to those in \cite{Tanguy2020,P2025,Blecher2024geometricinfluencesquantumboolean}.

While high-order versions of the influence inequality \eqref{eq:Tal-KKL} have been obtained, the high-order extensions of other energy-type inequalities --- such as Talagrand's isoperimetric-type inequality \eqref{eq:Tal-iso1} and the Eldan--Gross inequality \eqref{eq:Eldan-Gross} --- have remained unknown.

\subsection{Notations and settings}\label{subsec:settings}

Before stating the main theorems, here we first fix some basic notations and settings used throughout this paper.

We write $X \gtrsim Y$ (resp. $X \lesssim Y$) to denote $X \geq C Y$ (resp. $X \leq C Y$) for an absolute constant $C > 0$ that is independent of all underlying parameters (such as the dimension $n$, the derivative order $k$, and the space exponents $p, q$). We write $X \asymp Y$ if both $X \gtrsim Y$ and $X \lesssim Y$ hold. When the hidden constant depends on a specific parameter $w$, we append it as a subscript, writing $\gtrsim_w$, $\lesssim_w$, and $\asymp_w$.

Let $\Nbb := \{0, 1, \ldots\}$ denote the set of natural numbers, and let $\Rbb$ denote the set of real numbers. Throughout, $n$ represents an arbitrary positive integer, which typically signifies the dimension of the underlying spaces. We also use the standard shorthand $[n] := \{1, \ldots, n\}$ and $\Nbb_{\leq n} := \{0, 1, \ldots, n\}$.

The notation $|\cdot|$ is used in three distinct ways depending on the context: for a vector $x\in \Rbb^n$, we let $|x|:=\sqrt{\sum_i |x_i|^2}$ denote its Euclidean modulus (reducing to the usual absolute value when $n=1$); for a multi-index $\alpha\in \Nbb^n$, we let $|\alpha|:=\sum_i \alpha_i$ denote the sum of its components; for a subset $S\subseteq [n]$, we let $|S|$ denote its cardinality.

Two fundamental settings are considered in this paper: Boolean and Gaussian. In the Boolean setting, we consider the $n$-dimensional Boolean cube $\cube^n$ equipped with the uniform probability measure $\sigma_n$; in the Gaussian setting, we consider the $n$-dimensional Euclidean space $\Rbb^n$ equipped with the Gaussian measure $\gamma_n(\dif x)= \varphi_n(x)\dif x$, where $\varphi_n(x):= (2 \uppi)^{-n/2} \upe^{-\left|x\right|^{2}/2}$ denotes the standard Gaussian density. In both settings, we let $\Eop$ denote expectation, $\Var$ variance, $\left\|\,\cdot\,\right\|_{\Leb^p}$ the $\Leb^p$ norm, and $\langle\cdot,\cdot\rangle_{\Leb^2}$ the $\Leb^2$ inner product, all taken w.r.t. the underlying probability measure.

In the Boolean setting, the $i$-th discrete derivative of $f:\cube^n\to \Rbb$ is defined by
\begin{equation*}
    \Dop_i f(x) := \frac{f(x^{(i\mapsto 1)})-f(x^{(i\mapsto -1)})}{2},\quad x\in \cube^n,
\end{equation*}
where $x^{(i\mapsto \pm 1)}$ denotes the Boolean vector obtained via replacing the $i$-th coordinate of $x$ by $\pm 1$. The high-order discrete derivatives are then defined as follows: for nonempty $J\subseteq [n]$, we define $\Dop_J:=\prod_{i\in J} \Dop_i$; and for $J=\varnothing$, we adopt the convention $\Dop_{\varnothing}:=\Id$ (the identity operator). Moreover, for $k\in \Nbb_{\leq n}$, we define the $k$-th gradient tensor and its modulus by
\begin{equation*}
    \grad^k f:=(\Dop_J f)_{J\subseteq[n]:|J|=k},\quad \left|\grad^k f\right|^2 := \sum_{J\subseteq[n]:|J|=k} \left|\Dop_J f\right|^2.
\end{equation*}
For $f:\cube^n\to \Rbb$, we let $f^{\geq k}$ denote the component of $f$ orthogonal to $\Nul(\grad^k)$ (the null space of $\grad^k$) w.r.t. $\langle\cdot,\cdot\rangle_{\Leb^2}$, and write $\Wop^{\geq k}(f):= \left\|f^{\geq k}\right\|_{\Leb^2}^2$ (which coincides with the Fourier weight of $f$ of order at least $k$, see Section~\ref{subsec:Fourier} for more details). In particular, when $k=1$, $\grad$ annihilates only constant functions, and hence $\Wop^{\geq 1}(f)=\Var(f)$; when $k=0$, by convention $\grad^0 = \Id$, and hence $\Wop^{\geq 0}(f) = \left\|f\right\|_{\Leb^2}^2$.

In the Gaussian setting, we consider the \textit{normalized} high-order (weak) derivatives rather than the usual ones. Let $\Del_i:=\frac{\partial}{\partial x_i}$ denote the usual (weak) partial derivative operator over the $i$-th coordinate in $\Rbb^n$. For nonzero $\kappa\in \Nbb^n$, we define $\Dop^{\kappa}:= \prod_{i\in [n]} \frac{1}{\sqrt{\kappa_i !}} \Del_i^{\kappa_i}$; and for $\kappa=\ovec$, we adopt the convention $\Dop^{\ovec}:=\Id$. Moreover, for $k\in \Nbb$, we define the \textit{normalized} $k$-th (weak) gradient tensor and its modulus by
\begin{equation*}
    \grad^k f:=(\Dop^{\kappa} f)_{\kappa\in \Nbb^n:|\kappa|=k},\quad \left|\grad^k f\right|^2 := \sum_{\kappa\in \Nbb^n:|\kappa|=k} \left|\Dop^{\kappa} f\right|^2,
\end{equation*}
and let $\Sob^k(\Rbb^n,\gamma_n):=\left\{f\in \Leb^2(\Rbb^n,\gamma_n): \Dop^\kappa f \in \Leb^2(\Rbb^n,\gamma_n)\text{ for all } |\kappa|\leq k \right\}$ denote the usual Gaussian Sobolev space of order $k$. For $f\in \Leb^2(\Rbb^n,\gamma_n)$, we similarly let $f^{\geq k}$ denote the component of $f$ orthogonal to $\Nul(\grad^k)$ and write $\Wop^{\geq k}(f):= \left\|f^{\geq k}\right\|_{\Leb^2}^2$. We also have $\Wop^{\geq 1}(f)=\Var(f)$ and $\Wop^{\geq 0}(f) = \left\|f\right\|_{\Leb^2}^2$.
\begin{remark}
Here the modulus of the normalized $k$-th gradient tensor is also proportional to the usual one $\grad_{\text{usual}}^k f:=(\Del_{i_1}\ldots \Del_{i_k} f)_{i_1,i_2\ldots,i_k\in [n]}$,
\begin{align*}
    \left|\grad_{\text{usual}}^k f\right|^2 &= \sum_{i_1,\ldots,i_k\in [n]} \left| \Del_{i_1} \cdots \Del_{i_k} f \right|^2\\
    &= \sum_{\kappa\in \Nbb^n:|\kappa|=k} \binom{k}{\kappa_1,\ldots, \kappa_n} \left|\Del_1^{\kappa_1} \cdots \Del_n^{\kappa_n} f\right|^2 = k! \left|\grad^k f\right|^2.
\end{align*}
Compared to the usual high-order derivatives, the use of normalized ones provides a more concise presentation of results in the Gaussian setting, with algebraic structural properties and constants exactly aligned to the Boolean case.
\end{remark}

\subsection{Our results}\label{subsec:results}

Classical Talagrand-type inequalities have become cornerstones of Boolean analysis and high-dimensional probability. Questions regarding their high-order analogues --- relating high-order energies $\left\|\left|\grad^k f\right|\right\|_{\Leb^p}^{p}$ to high-order Fourier weights $\Wop^{\geq k}(f)$ --- arise naturally when one studies functions with $k$-wise interactions or when one seeks sharp concentration beyond the first Fourier level. Despite considerable interest, a systematic high-order theory has been missing; in particular, the high-order analogues of the isoperimetric-type and Eldan--Gross inequalities have remained open.

In this paper, we refine the approach in \cite{ChangLi2026} and develop a unified \textit{high-order variance-decay framework} for proving high-order generalizations of the aforementioned Talagrand-type inequalities. Using this framework, we establish high-order Poincar\'e-type inequalities, $\Leb^p$--$\Leb^q$ inequalities, isoperimetric-type inequalities, and Falik--Samorodnitsky \& Eldan--Gross inequalities, all with sharp logarithmic powers and explicit constants.

\subsubsection*{High-order Poincar\'{e}-type inequalities}

Our first result yields high-order Poincar\'{e}-type inequalities in both Boolean and Gaussian settings, establishing direct relations between high-order energies and Fourier weights. While these inequalities are quantitatively weaker than the other energy-type inequalities presented below, they have the advantage of sharper constants.

More specifically, in the Boolean setting, we prove the following:
\begin{theorem}\label{thm:high-order-Poincare-type-Boolean}
For all $k\in \Nbb_{\leq n}$, $f:\cube^n\to [-1,1]$, and $p\in [1,2]$,
\begin{equation}\label{eq:high-order-Poincare-type-Boolean}
    \left\|\left|\grad^k f\right|\right\|_{\Leb^p}^{p} \geq C^{\text{\rm\ref{thm:high-order-Poincare-type-Boolean}}}(p)^k\Wop^{\geq k}(f),
\end{equation}
with $C^{\text{\rm\ref{thm:high-order-Poincare-type-Boolean}}}(p)=2^{1-p}/\Beta(\frac{1}{2};\frac{p}{2},\frac{p}{2})$, where $\Beta(x;a,b) := \int_0^x t^{a-1} \left(1-t\right)^{b-1} \dif t$ is the incomplete Beta function.
\end{theorem}
\begin{remark}
The explicit expression ensures that $C^{\text{\rm\ref{thm:high-order-Poincare-type-Boolean}}}(p)$ has a uniform positive lower bound. When $p=2$, $C^{\text{\rm\ref{thm:high-order-Poincare-type-Boolean}}}(2)=1$, and \eqref{eq:high-order-Poincare-type-Boolean} yields the canonical high-order version of the Poincar\'e inequality \eqref{eq:boolean-Poincare}. When $p=1$, $C^{\text{\rm\ref{thm:high-order-Poincare-type-Boolean}}}(1)=2/\pi$, which matches the best known explicit constant for the $\Leb^1$ Poincar\'e inequality on the Boolean cube (see \cite{EL2008,Ivanisvili2019improvingconstantendpointpoincare}).
\end{remark}

In the Gaussian setting, we similarly have:
\begin{theorem}\label{thm:high-order-Poincare-type-Gaussian}
For all $k\in \Nbb$, $f\in \Sob^k(\Rbb^n,\gamma_n)$ with $\left\|f\right\|_{\Leb^\infty}\leq 1$, and $p\in [1,2]$,
\begin{equation}\label{eq:high-order-Poincare-type-Gaussian}
    \left\|\left|\grad^k f\right|\right\|_{\Leb^p}^{p} \geq C^{\text{\rm\ref{thm:high-order-Poincare-type-Gaussian}}}(p)^k\Wop^{\geq k}(f),
\end{equation}
with $C^{\text{\rm\ref{thm:high-order-Poincare-type-Gaussian}}}(p) = C^{\text{\rm\ref{thm:high-order-Poincare-type-Boolean}}}(p) \gtrsim 1$.
\end{theorem}

\subsubsection*{High-order $\Leb^p$--$\Leb^q$ inequalities}

Next, we present high-order $\Leb^p$--$\Leb^q$ generalizations of Talagrand's influence inequality \eqref{eq:Tal-KKL} together with energy-type analogues in both Boolean and Gaussian settings.

More specifically, in the Boolean setting, we establish the following high-order $\Leb^p$--$\Leb^q$ influence inequality:
\begin{theorem}\label{thm:high-order-Lp-Lq-inf-Boolean}
For all $k\in \Nbb_{\leq n}$, $f:\cube^n\to [-1,1]$, $p\in [1,2]$ and $q\in [1,2)$,
\begin{equation}\label{eq:high-order-Lp-Lq-inf-Boolean}
    \sum_{J\subseteq[n]:|J|=k}\frac{\left\|\Dop_J f\right\|_{\Leb^p}^{p}}{\big[1+\frac{q}{2k(2-q)} \ln^{+}(\left\|\Dop_J f\right\|_{\Leb^p}^{p}/\left\|\Dop_J f\right\|_{\Leb^q}^{2})\big]^{k}} \geq C^{\text{\rm\ref{thm:high-order-Lp-Lq-inf-Boolean}}}(q)^k \Wop^{\geq k}(f),
\end{equation}
where $C^{\text{\rm\ref{thm:high-order-Lp-Lq-inf-Boolean}}}(q) = \big(1+\frac{q}{2\upe(2-q)}\big)^{-1} \gtrsim 2-q$.
\end{theorem}
\begin{remark}
Here \eqref{eq:high-order-Lp-Lq-inf-Boolean} attains an optimal logarithmic power of $k$ compared to the quantum high-order influence inequality in \cite[Corollary 4.5]{ChangLi2026}. When $p=2$ and $q=1$, \eqref{eq:high-order-Lp-Lq-inf-Boolean} yields a $k$-th order generalization of Talagrand's influence inequality \eqref{eq:Tal-KKL} with constant $C^{\text{\rm\ref{thm:high-order-Lp-Lq-inf-Boolean}}}(1)\geq 0.844$. In particular, when restricted to $f:\cube^n\to \cube$, \eqref{eq:high-order-Lp-Lq-inf-Boolean} further yields the high-order influence inequality in \cite{P2025}:
\begin{equation}
    \sum_{J\subseteq[n]:|J|=k}\frac{\Inf_J(f)}{\big[1+\ln(1/\Inf_J(f))\big]^{k}} \geq \left(\frac{C^{\text{\rm\ref{thm:high-order-Lp-Lq-inf-Boolean}}}(1)}{2k}\right)^{k} \Wop^{\geq k}(f),
\end{equation}
where $\Inf_J(f):= \left\|\Dop_J f\right\|_{\Leb^2}^2$ is the high-order influence of $J$. Note that for $\cube$-valued $f$, we have $\left|\Dop_J f\right|\in 2^{1-|J|} \Nbb$, and thus $\left\|\Dop_J f\right\|_{\Leb^1} \leq 2^{|J|-1} \Inf_J(f)$.
\end{remark}

Analogously, we also prove a high-order $\Leb^p$--$\Leb^q$ energy inequality:
\begin{theorem}\label{thm:high-order-Lp-Lq-energy-Boolean}
For all $k\in \Nbb_{\leq n}$, $f:\cube^n\to [-1,1]$, and $1\leq q\leq p\leq 2$,
\begin{equation}\label{eq:high-order-Lp-Lq-energy-Boolean}
    \frac{\left\|\left|\grad^k f\right|\right\|_{\Leb^p}^{p}}{\big[1+\frac{q}{2k(2-q)} \ln^{+}(\left\|\left|\grad^k f\right|\right\|_{\Leb^p}^{p}/\left\|\left|\grad^k f\right|\right\|_{\Leb^q}^{2})\big]^{kp/2}} \geq C^{\text{\rm\ref{thm:high-order-Lp-Lq-energy-Boolean}}}(p,q)^k \Wop^{\geq k}(f),
\end{equation}
with $C^{\text{\rm\ref{thm:high-order-Lp-Lq-energy-Boolean}}}(p,q) = \left(M(p)^{-\frac{2}{p}}+\frac{pq}{4\upe(2-q)}\right)^{-\frac{p}{2}}  \gtrsim 2-q$, where $M(p) := p \, 2^{-\frac{p}{2}} \max_{x\geq 0} \frac{1-\upe^{-x}}{x^{p/2}}$.
\end{theorem}
\begin{remark}
The explicit expression ensures that $M(p)$ also has a uniform positive lower bound. Here \eqref{eq:high-order-Lp-Lq-energy-Boolean} provides a classical high-order analogue of the first-order quantum $\Leb^p$--$\Leb^q$ energy inequality in \cite[Corollary 4.4]{ChangLi2026}. Compared to the influence inequality \eqref{eq:high-order-Lp-Lq-inf-Boolean}, the energy inequality \eqref{eq:high-order-Lp-Lq-energy-Boolean} (and also other energy-type inequalities presented below) has logarithmic power $kp/2$ rather than $k$. This difference stems from the fact that discrete derivatives are bounded ($\left\|\Dop_J\right\|_{\Leb^\infty\to \Leb^\infty}\leq 1$), whereas the gradients only satisfy a dimension-free Bismut-type bound $\left\|\left|\grad^k \Pop_t f\right|\right\|_{\Leb^\infty} \lesssim t^{-k/2}\left\|f\right\|_{\Leb^\infty}$ under the smoothing of the noise semigroup $\Pop_t$ (see Section \ref{sec:preliminaries} for more details).
\end{remark}

In the Gaussian setting, we also have the high-order energy and influence inequalities:
\begin{theorem}\label{thm:high-order-Lp-Lq-Gaussian}
For all $k\in \Nbb$, $f\in \Sob^k(\Rbb^n,\gamma_n)$ with $\left\|f\right\|_{\Leb^\infty}\leq 1$, and $1\leq q\leq p\leq 2$,
\begin{equation}\label{eq:high-order-Lp-Lq-energy-Gaussian}
    \frac{\left\|\left|\grad^k f\right|\right\|_{\Leb^p}^{p}}{\big[1+\frac{q}{2k(2-q)} \ln^{+}(\left\|\left|\grad^k f\right|\right\|_{\Leb^p}^{p}/\left\|\left|\grad^k f\right|\right\|_{\Leb^q}^{2})\big]^{kp/2}} \geq C^{\text{\rm\ref{thm:high-order-Lp-Lq-Gaussian}}}(p,q)^k \Wop^{\geq k}(f),
\end{equation}
and further,
\begin{equation}\label{eq:high-order-Lp-Lq-inf-Gaussian}
    \sum_{\kappa\in \Nbb^n:|\kappa|=k}\frac{\left\|\Dop^\kappa f\right\|_{\Leb^p}^{p}}{\big[1+\frac{q}{2k(2-q)} \ln^{+}(\left\|\Dop^\kappa f\right\|_{\Leb^p}^{p}/\left\|\Dop^\kappa f\right\|_{\Leb^q}^{2})\big]^{kp/2}} \geq C^{\text{\rm\ref{thm:high-order-Lp-Lq-Gaussian}}}(p,q)^k \Wop^{\geq k}(f),
\end{equation}
with $C^{\text{\rm\ref{thm:high-order-Lp-Lq-Gaussian}}}(p,q) = C^{\text{\rm\ref{thm:high-order-Lp-Lq-energy-Boolean}}}(p,q) \gtrsim 2-q$.
\end{theorem}
\begin{remark}
Here the logarithmic power $kp/2$ in the Gaussian influence inequality \eqref{eq:high-order-Lp-Lq-inf-Gaussian} relies similarly on $\left\|\Dop^\kappa \Pop_t f\right\|_{\Leb^\infty} \lesssim t^{-|\kappa|/2}\left\|f\right\|_{\Leb^\infty}$ under the Ornstein--Uhlenbeck semigroup $\Pop_t$. When $q=1$, \eqref{eq:high-order-Lp-Lq-inf-Gaussian} provides a high-order refinement of the $\Leb^p$--$\Leb^1$ influence inequality in \cite[Theorem 6]{CL2012hypercontr}, and further when $p=2$, \eqref{eq:high-order-Lp-Lq-inf-Gaussian} recovers the high-order $\Leb^2$--$\Leb^1$ influence inequality in \cite[Theorem 12]{Tanguy2020} with constant $C^{\text{\rm\ref{thm:high-order-Lp-Lq-Gaussian}}}(2,1)\geq 0.844$.
\end{remark}

\subsubsection*{High-order isoperimetric-type inequalities}

Using the high-order $\Leb^p$--$\Leb^q$ inequalities, we further derive high-order $\Leb^p$ extensions of Talagrand's isoperimetric-type inequality \eqref{eq:Tal-iso1} together with KKL-type analogues in both Boolean and Gaussian settings.

More specifically, in the Boolean setting, the following high-order isoperimetric-type inequality is deduced from the high-order energy inequality \eqref{eq:high-order-Lp-Lq-energy-Boolean} in Theorem \ref{thm:high-order-Lp-Lq-energy-Boolean}:
\begin{theorem}\label{thm:high-order-Tal-isoper-Boolean}
For all $k\in \Nbb_{\leq n}$, $f:\cube^n\to [-1,1]$, and $p\in [1,2)$,
\begin{equation}\label{eq:high-order-Tal-isoper-Boolean}
    \left\|\left|\grad^k f\right|\right\|_{\Leb^p}^{p} \geq C^{\text{\rm\ref{thm:high-order-Tal-isoper-Boolean}}}(p)^k\Wop^{\geq k}(f) \left[1+ \frac{1}{2k}\ln\left(\frac{1}{\Wop^{\geq k}(f)}\right)\right]^\frac{kp}{2},
\end{equation}
with $C^{\text{\rm\ref{thm:high-order-Tal-isoper-Boolean}}}(p) = \left(M(p)^{-\frac{2}{p}}+\frac{p}{2\upe(2-p)}\right)^{-\frac{p}{2}}  \gtrsim 2-p$.
\end{theorem}
\begin{remark}
When $p=1$, \eqref{eq:high-order-Tal-isoper-Boolean} yields a $k$-th order generalization of Talagrand's isoperimetric-type inequality \eqref{eq:Tal-iso1} with constant $C^{\text{\rm\ref{thm:high-order-Tal-isoper-Boolean}}}(1)\geq 0.443$. When $p\to 2$, the constant $C^{\text{\rm\ref{thm:high-order-Tal-isoper-Boolean}}}(p)\to 0$, and the inequality becomes trivial when $p=2$. This is natural because an $\Leb^2$ isoperimetric-type inequality with a positive constant cannot hold for all real-valued functions: when $n=1$ and $k=1$, $\left\|\left|\grad f\right|\right\|_{\Leb^2}^{2} = \Var(f)$, and clearly
\begin{equation*}
    \Var(f) \gtrsim \Var(f) \ln\left(\frac{1}{\Var(f)}\right)
\end{equation*}
fails for $f$ with sufficiently small variance.
\end{remark}

Analogously, we also obtain the following high-order KKL-type isoperimetric inequality using the high-order influence inequality \eqref{eq:high-order-Lp-Lq-inf-Boolean} in Theorem \ref{thm:high-order-Lp-Lq-inf-Boolean}:
\begin{theorem}\label{thm:high-order-partial-isoper-Boolean}
For all $k\in \Nbb_{\leq n}$, $f:\cube^n\to [-1,1]$, and $p\in [1,2)$,
\begin{equation}\label{eq:high-order-partial-isoper-Boolean}
    \max_{J\subseteq[n]:|J|=k}\left\|\Dop_J f\right\|_{\Leb^p}^{p} \geq C^{\text{\rm\ref{thm:high-order-partial-isoper-Boolean}}}(p)^k\, \frac{\Wop^{\geq k}(f)}{\binom{n}{k}} \left[1+ \frac{1}{2k} \ln\left(\frac{\binom{n}{k}}{\Wop^{\geq k}(f)}\right)\right]^k,
\end{equation}
where $C^{\text{\rm\ref{thm:high-order-partial-isoper-Boolean}}}(p) = \left(1+\frac{1}{\upe(2-p)}\right)^{-1} \gtrsim 2-p$.
\end{theorem}
\begin{remark}
Here \eqref{eq:high-order-partial-isoper-Boolean} attains an optimal logarithmic power of $k$ compared to the quantum high-order partial isoperimetric inequality in \cite[Corollary 4.7]{ChangLi2026}, providing an isoperimetric-type strengthening of the following high-order $\Leb^p$ KKL inequality:
\begin{equation}\label{eq:high-order-Lp-KKL-Boolean}
    \max_{J\subseteq[n]:|J|=k}\left\|\Dop_J f\right\|_{\Leb^p}^{p} \gtrsim_p \Wop^{\geq k}(f) \left(\frac{\ln(n)}{n}\right)^k.
\end{equation}
When restricted to $\cube$-valued $f$, \eqref{eq:high-order-partial-isoper-Boolean} also yields the high-order KKL inequality in \cite{P2025}. See Section \ref{subsec:high-order-iso-proof} for more details.
\end{remark}

In the Gaussian setting, the high-order energy and influence inequalities in Theorem~\ref{thm:high-order-Lp-Lq-Gaussian} similarly yield the following isoperimetric-type inequalities:
\begin{theorem}\label{thm:high-order-iso-Gaussian}
For all $k\in \Nbb$, $f\in \Sob^k(\Rbb^n,\gamma_n)$ with $\left\|f\right\|_{\Leb^\infty}\leq 1$, and $p\in [1,2)$,
\begin{equation}\label{eq:high-order-iso-Gaussian}
    \left\|\left|\grad^k f\right|\right\|_{\Leb^p}^{p} \geq C^{\text{\rm\ref{thm:high-order-iso-Gaussian}}}(p)^k\Wop^{\geq k}(f) \left[1+ \frac{1}{2k}\ln\left(\frac{1}{\Wop^{\geq k}(f)}\right)\right]^\frac{kp}{2},
\end{equation}
and further,
\begin{equation}\label{eq:high-order-partial-iso-Gaussian}
    \max_{\kappa\in \Nbb^n:|\kappa|=k}\left\|\Dop^\kappa f\right\|_{\Leb^p}^{p} \geq C^{\text{\rm\ref{thm:high-order-iso-Gaussian}}}(p)^k\, \frac{\Wop^{\geq k}(f)}{\binom{n+k-1}{k}} \left[1+ \frac{1}{2k} \ln\left(\frac{\binom{n+k-1}{k}}{\Wop^{\geq k}(f)}\right)\right]^{\frac{kp}{2}},
\end{equation}
with $C^{\text{\rm\ref{thm:high-order-iso-Gaussian}}}(p) = C^{\text{\rm\ref{thm:high-order-Tal-isoper-Boolean}}}(p) \gtrsim 2-p$.
\end{theorem}
\begin{remark}
Here \eqref{eq:high-order-partial-iso-Gaussian} yields a high-order $\Leb^p$ functional extension of the geometric influence inequality in \cite[Corollary 7]{CL2012hypercontr}, and further implies a Gaussian high-order $\Leb^p$ KKL inequality
\begin{equation}\label{eq:high-order-Lp-KKL-Gaussian}
    \max_{\kappa\in \Nbb^n:|\kappa|=k}\left\|\Dop^\kappa f\right\|_{\Leb^p}^{p} \gtrsim_p \Wop^{\geq k}(f) \left(\frac{\left(\ln(n+k-1)\right)^{\frac{p}{2}}}{n+k-1}\right)^k.
\end{equation}
See Section \ref{subsec:high-order-iso-proof} for more details.
\end{remark}

\subsubsection*{High-order Falik--Samorodnitsky \& Eldan--Gross inequalities}

Finally, we establish high-order extensions of the Falik--Samorodnitsky inequality in \cite{Falik-Samorodnitsky2007Edge-Isoperimetric} and the Eldan--Gross inequality \eqref{eq:Eldan-Gross}.

More specifically, in the Boolean setting, we prove the following:
\begin{theorem}\label{thm:high-order-Eldan-Gross-Boolean}
For all $k\in \Nbb_{\leq n}$, $f:\cube^n\to \Rbb$ and $q\in [1,2)$,
\begin{equation}\label{eq:high-order-Falik-Samorodnitsky-Boolean}
    \left\|\left|\grad^k f\right|\right\|_{\Leb^2}^{2} \geq \Wop^{\geq k}(f) \left[1+ \frac{q}{2k(2-q)}\ln^+\left(\frac{\Wop^{\geq k}(f)}{\sum_{J:|J|=k} \left\|\Dop_J f\right\|_{\Leb^q}^2} \right)\right]^k,
\end{equation}
and further for $\left\|f\right\|_{\Leb^\infty}\leq 1$ and $q\leq p<2$,
\begin{equation}\label{eq:high-order-Eldan-Gross-Boolean}
    \left\|\left|\grad^k f\right|\right\|_{\Leb^p}^{p} \geq C^{\text{\rm\ref{thm:high-order-Eldan-Gross-Boolean}}}(p,q)^k\Wop^{\geq k}(f) \left[1 + \frac{q}{2k(2-q)}\ln^+\left(\frac{1}{\sum_{J:|J|=k} \left\|\Dop_J f\right\|_{\Leb^q}^2} \right)\right]^\frac{kp}{2},
\end{equation}
with $C^{\text{\rm\ref{thm:high-order-Eldan-Gross-Boolean}}}(p,q) = \left(\frac{2}{2-q}\,M(p)^{-\frac{2}{p}}+\frac{pq}{2\upe(2-p)(2-q)}\right)^{-\frac{p}{2}} \gtrsim \left(2-p\right)\left(2-q\right)$.
\end{theorem}
\begin{remark}
Here \eqref{eq:high-order-Eldan-Gross-Boolean} provides a high-order extension of the Eldan--Gross inequality \eqref{eq:Eldan-Gross} to all bounded real-valued functions for $p\in [1,2)$. When $p=2$, the Eldan--Gross inequality does not hold for all real-valued functions (similar to the isoperimetric case in Theorem~\ref{thm:high-order-Tal-isoper-Boolean}); instead, we have the high-order Falik--Samorodnitsky inequality \eqref{eq:high-order-Falik-Samorodnitsky-Boolean}, which recovers (and in fact strengthens, by an additive constant $1$) the classical one when $k=1$ and $q=1$.
\end{remark}

In the Gaussian setting, we similarly have:
\begin{theorem}\label{thm:high-order-Eldan-Gross-Gaussian}
For all $k\in \Nbb$, $f\in \Sob^k(\Rbb^n,\gamma_n)$ and $q\in [1,2)$,
\begin{equation}\label{eq:high-order-Falik-Samorodnitsky-Gaussian}
    \left\|\left|\grad^k f\right|\right\|_{\Leb^2}^{2} \geq \Wop^{\geq k}(f) \left[1+ \frac{q}{2k(2-q)}\ln^+\left(\frac{\Wop^{\geq k}(f)}{\sum_{\kappa:|\kappa|=k} \left\|\Dop^\kappa f\right\|_{\Leb^q}^2} \right)\right]^k,
\end{equation}
and further for $\left\|f\right\|_{\Leb^\infty}\leq 1$ and $q\leq p<2$,
\begin{equation}\label{eq:high-order-Eldan-Gross-Gaussian}
    \left\|\left|\grad^k f\right|\right\|_{\Leb^p}^{p} \geq C^{\text{\rm\ref{thm:high-order-Eldan-Gross-Gaussian}}}(p,q)^k\Wop^{\geq k}(f) \left[1 +  \frac{q}{2k(2-q)}\ln^+\left(\frac{1}{\sum_{\kappa:|\kappa|=k} \left\|\Dop^\kappa f\right\|_{\Leb^q}^2} \right)\right]^\frac{kp}{2},
\end{equation}
with $C^{\text{\rm\ref{thm:high-order-Eldan-Gross-Gaussian}}}(p,q) = C^{\text{\rm\ref{thm:high-order-Eldan-Gross-Boolean}}}(p,q) \gtrsim \left(2-p\right)\left(2-q\right)$.
\end{theorem}

\subsection{Proof outline: the high-order variance-decay framework}\label{subsec:outline}

To illuminate the core analytical mechanism, we outline here the high-order variance-decay framework used to establish our main results in Section~\ref{subsec:results}. For clarity of exposition, we restrict our focus to the case $p=1$ and a bounded function $f: \left\|f\right\|_{\Leb^\infty} \lesssim 1$.

To appreciate the high-order obstruction, let us first recall the first-order variance-decay framework established in \cite{ChangLi2026}. When $k=1$, we have the Bismut-type gradient estimate and a subsequent interpolation bound under the smoothing of the associated ergodic symmetric Markov semigroup $\Pop_t$:
\begin{equation*}
    \left\|\left|\grad \Pop_t f\right|\right\|_{\Leb^\infty} \lesssim t^{-1/2} \implies \left\|\left|\grad \Pop_t f\right|\right\|_{\Leb^2}^2 \lesssim t^{-1/2} \left\|\left|\grad f\right|\right\|_{\Leb^1}.
\end{equation*}
Combined with the variance representation formula:
\begin{equation}\label{eq:semigroup-var-representation}
    \Var(f) = \int_{0}^{\infty} 2 \left\|\left|\grad \Pop_s f\right|\right\|_{\Leb^2}^2 \dif s,
\end{equation}
every sharp short-time variance-decay estimate with rate $\Rscr\gtrsim 1$ and time horizon $\epsilon>0$ on the given function $f$ along the semigroup $\Pop_t$,
\begin{equation}\label{eq:var-decay-abstract}
    \Var(\Pop_t f) \leq \upe^{-2t \Rscr} \Var(f),\quad \forall\, t\in [0,\epsilon],
\end{equation}
naturally yields
\begin{align*}
    \left(1-\upe^{-2 t \Rscr}\right) \Var(f) &\leq \Var(f)-\Var(\Pop_t f) \\
    &= \int_{0}^{t} 2 \left\|\left|\grad \Pop_s f\right|\right\|_{\Leb^2}^2 \dif s \lesssim t^{1/2} \left\|\left|\grad f\right|\right\|_{\Leb^1},
\end{align*}
and hence, by optimizing over $t \in [0, \epsilon]$, leads to the following Talagrand-type inequality corresponding to the rate $\Rscr$:
\begin{equation*}
    \left\|\left|\grad f\right|\right\|_{\Leb^1} \gtrsim_\epsilon \Var(f)\, \Rscr^{1/2}.
\end{equation*}
The task of proving Talagrand-type inequalities is thus reduced to establishing the variance-decay estimates \eqref{eq:var-decay-abstract} via hypercontractivity.

When attempting to lift this elegant framework to higher orders $k > 1$, a fundamental analytical obstruction emerges. The corresponding gradient estimates scale with the order (see Section~\ref{subsec:derivative-estimate} for more details):
\begin{equation*}
    \left\|\left|\grad^k \Pop_t f\right|\right\|_{\Leb^\infty} \lesssim_k t^{-k/2} \implies \left\|\left|\grad^k \Pop_t f\right|\right\|_{\Leb^2}^2 \lesssim_k t^{-k/2} \left\|\left|\grad^k f\right|\right\|_{\Leb^1}.
\end{equation*}
If one attempts a naive extension of the variance representation \eqref{eq:semigroup-var-representation} by directly integrating the high-order energy $\left\|\left|\grad^k \Pop_s f\right|\right\|_{\Leb^2}^2$, the previous short-time argument collapses, as $t^{-k/2}$ is \emph{not integrable} near the origin $t=0$ for $k > 1$.

Our key insight is that this non-integrability can be structurally fixed by rebalancing the integration powers in the fashion of Besov norms. To this end, we introduce the \textit{high-order Besov-type variance functional}
\begin{equation*}
    \Vop^k(f) := \left(\int_0^\infty 2 \left\|\left|\grad^k \Pop_s f\right|\right\|_{\Leb^2}^{2/k} \dif s\right)^{k}\asymp_k \left\|f^{\geq k}\right\|_{\Besov_{2,2/k}^0}^2,
\end{equation*}
which behaves as the squared equivalent Besov $\Besov_{2,2/k}^0$ quasi-seminorm and recovers the canonical variance when $k=1$. Now, the short-time singularity is precisely regularized: assuming a similar decay estimate of the high-order variance functional,
\begin{equation}\label{eq:high-order-var-decay-abstract}
    \Vop^k(\Pop_t f) \leq \upe^{-2t \Rscr} \Vop^k(f),\quad \forall\, t\in [0,\epsilon],
\end{equation}
we then obtain
\begin{align*}
    \left(1-\upe^{-2 t \Rscr/k }\right) \Vop^k(f)^{1/k} &\leq \Vop^k(f)^{1/k} - \Vop^k(\Pop_t f)^{1/k} \\
    &= \int_{0}^{t} 2 \left\|\left|\grad^k \Pop_s f\right|\right\|_{\Leb^2}^{2/k} \dif s \lesssim_k t^{1/2} \left\|\left|\grad^k f\right|\right\|_{\Leb^1}^{1/k},
\end{align*}
which leads to a corresponding high-order Talagrand-type inequality:
\begin{equation*}
    \left\|\left|\grad^k f\right|\right\|_{\Leb^1} \gtrsim_{\epsilon,k} \Vop^k(f)\left(\Rscr/k\right)^{k/2}.
\end{equation*}
Crucially, this Besov-type rebalancing of the integration power also preserves our ability of establishing the high-order variance-decay estimates \eqref{eq:high-order-var-decay-abstract} via hypercontractivity (see Section~\ref{subsec:high-order-var-decay} for more details).

Finally, we connect these functional bounds to the underlying Fourier spectrum. Indeed, since $\Besov_{2,2/k}^0\subseteq \Leb^2$ for $k\geq 1$, the Besov-type nature of $\Vop^k(f)$ yields
\begin{equation*}
    \Vop^k(f) \asymp_k \left\|f^{\geq k}\right\|_{\dot{\Besov}_{2,2/k}^0}^2 \gtrsim_k \left\|f^{\geq k}\right\|_{\Leb^2}^2 = \Wop^{\geq k}(f).
\end{equation*}
More fundamentally, this connection is realized via an exact, sharp lower bound: $\Vop^k(f) \geq \Wop^{\geq k}(f)$ (see Section~\ref{subsec:high-order-variance-compare}). This inequality directly anchors our high-order variance-decay framework to the high-order Fourier weights, immediately giving rise to the comprehensive catalog of high-order results detailed in Section~\ref{subsec:results}.

\section{Preliminaries}\label{sec:preliminaries}

In this section, we collect the key analytic tools in both the Boolean and Gaussian settings. For a comprehensive introduction we refer to \cite{Ryan2014book}.

\subsection{Fourier analysis} \label{subsec:Fourier}

We begin by reviewing the classical Fourier frameworks, which provide convenient spectral representations that simplify the analysis in both the Boolean and Gaussian settings.

In the Boolean setting, an orthonormal basis of $\Leb^{2}(\cube^n,\sigma_n)$ is given by the \textit{Walsh functions} $\chi_S(x) := \prod_{i\in S} x_i$, $S\subseteq [n]$. Thus, every function $f:\cube^n\to \Rbb$ admits a unique Fourier--Walsh expansion
\begin{equation*}
    f = \sum_{S\subseteq[n]} \widehat{f}(S) \chi_S,\quad \widehat{f}(S) := \langle f,\chi_S \rangle_{\Leb^2}.
\end{equation*}
The discrete derivatives act as formal differentiations on these Fourier--Walsh expansions:
\begin{equation}\label{eq:derivative-Fourier-Boolean}
    \Dop_J f = \sum_{S\subseteq[n]:S\supseteq J} \widehat{f}(S) \chi_{S\setminus J}.
\end{equation}
In particular, for $k\in \Nbb_{\leq n}$, the null space of the $k$-th gradient tensor is given by
\begin{equation*}
    \Nul(\grad^k)=\opspan\left\{\chi_S: |S|<k\right\}.
\end{equation*}
Consequently, the high-order components and weights defined in Section~\ref{subsec:settings} satisfy
\begin{equation*}
    f^{\geq k} = \sum_{S\subseteq[n]: |S|\geq k} \widehat{f}(S) \chi_S, \quad \Wop^{\geq k}(f) = \sum_{S\subseteq[n]: |S|\geq k} \widehat{f}(S)^2,
\end{equation*}
which coincide with the usual definitions of high-order Fourier components and weights. We also use the notation $\Wop^{k}(f) := \sum_{S: |S|= k} \widehat{f}(S)^2$.

In the Gaussian setting, an orthonormal basis of $\Leb^{2}(\Rbb^n,\gamma_n)$ is given by the \textit{normalized multivariate Hermite polynomials} $h_{\alpha}:=\left(-1\right)^{|\alpha|} \varphi_n^{-1} \Dop^{\alpha} \varphi_n$, $\alpha\in \Nbb^n$. Analogously, every function $f\in \Leb^{2}(\Rbb^n,\gamma_n)$ admits a unique Fourier--Hermite expansion
\begin{equation*}
    f = \sum_{\alpha\in \Nbb^n} \widehat{f}(\alpha) h_\alpha,\quad \widehat{f}(\alpha) := \langle f,h_\alpha \rangle_{\Leb^2}.
\end{equation*}
The normalized high-order derivatives then transform according to the relation
\begin{equation}\label{eq:derivative-Fourier-Gaussian}
    \Dop^{\kappa} f = \sum_{\alpha\in \Nbb^n:\alpha\geq \kappa} \sqrt{\binom{\alpha}{\kappa}}\, \widehat{f}(\alpha) h_{\alpha-\kappa}.
\end{equation}
where $\alpha\geq \kappa$ denotes componentwise inequality ($\alpha_i\geq \kappa_i$ for all $i\in [n]$), $\alpha-\kappa$ is the entrywise difference multi-index, and $\binom{\alpha}{\kappa}:= \prod_{i\in [n]} \binom{\alpha_i}{\kappa_i}$. For $k\in \Nbb$, the null space of the normalized $k$-th gradient tensor is similarly given by
\begin{equation*}
    \Nul(\grad^k)=\opspan\left\{h_{\alpha}: |\alpha|<k\right\},
\end{equation*}
yielding the spectral characterizations
\begin{equation*}
    f^{\geq k} = \sum_{\alpha\in \Nbb^n:|\alpha|\geq k} \widehat{f}(\alpha) h_\alpha,\quad \Wop^{\geq k}(f) = \sum_{\alpha\in \Nbb^n: |\alpha|\geq k} \widehat{f}(\alpha)^2,
\end{equation*}
matching the structural Fourier components and weights of order at least $k$. We similarly write $\Wop^{k}(f) := \sum_{\alpha: |\alpha|= k} \widehat{f}(S)^2$.

\subsection{Semigroup analysis}

Semigroups play a central role in modern analysis, particularly in the study of functional inequalities, diffusion processes, and harmonic analysis. Here, the \textit{noise semigroup} in the Boolean setting and the \textit{Ornstein--Uhlenbeck (OU) semigroup} in the Gaussian setting, both denoted by $(\Pop_t)_{t\geq 0}$, are prototypical examples of \textit{ergodic symmetric Markov semigroups}. Moreover, they are compatible with the spatial derivatives, and enjoy the fundamental property of hypercontractivity. Together, these semigroups provide a powerful framework for smoothing functions and deriving sharp estimates.

\subsubsection*{Noise semigroup in the Boolean setting}
For $t\geq 0$ and $x\in \cube^n$, we consider a random vector $\ybold\sim p_{t,x}$ generated by flipping each coordinate of $x$ independently with probability $\frac{1-\upe^{-t}}{2}$. Its distribution then yields the transition probability of the noise semigroup:
\begin{equation*}
    p_t(x,y) := \Prob_{\ybold\sim p_{t,x}}\left(\ybold = y\right) = \prod_{i\in [n]} \frac{1+\upe^{-t}x_i y_i}{2}.
\end{equation*}
The noise operator $\Pop_t$ acting on functions $f:\cube^n\to \Rbb$ is then defined by
\begin{equation}\label{eq:def-semigroup-Boolean}
    \Pop_t f (x) := \Exp_{\ybold\sim p_{t,x}}\left[f(\ybold)\right] = \sum_{y\in \cube^n} f(y) \, p_t(x,y),\quad x\in \cube^n.
\end{equation}
The operators $(\Pop_t)_{t\geq 0}$ form an ergodic symmetric Markov semigroup that is diagonalized by the Walsh basis $\{\chi_S\}_{S \subseteq [n]}$, yielding the Fourier--Walsh expansion:
\begin{equation}\label{eq:semigroup-Fourier-Boolean}
    \Pop_t f = \sum_{S\subseteq[n]} \upe^{-t|S|}\widehat{f}(S) \chi_S.
\end{equation}

The noise semigroup commutes with the discrete derivatives as follows:
\begin{lemma}
\label{lem:derivative-semigroup-decay-Boolean}
For all $J\subseteq [n]$, $f:\cube^n\to \Rbb$ and $t\geq 0$,
\begin{equation}\label{eq:J-derivative-semigroup-decay-Boolean}
    \Dop_J \Pop_t f = \upe^{-t|J|} \Pop_t \Dop_J f,
\end{equation}
and consequently, for all $k\in \Nbb_{\leq n}$,
\begin{equation}\label{eq:k-derivative-semigroup-decay-Boolean}
    \left|\grad^k \Pop_t f\right| \leq \upe^{-kt}\Pop_t \left|\grad^k f\right|.
\end{equation}
\end{lemma}
\begin{remark}
For $p\in [1,\infty]$, combining with the $\Leb^p$ contraction property of the noise operator $\Pop_t$, we further obtain the energy-decay estimates:
\begin{equation}\label{eq:derivative-semigroup-decay-Boolean}
    \left\|\Dop_J \Pop_t f\right\|_{\Leb^p} \leq \upe^{-t|J|} \left\|\Dop_J f\right\|_{\Leb^p}, \quad \left\|\left|\grad^k \Pop_t f\right|\right\|_{\Leb^p} \leq \upe^{-kt} \left\|\left|\grad^k f\right|\right\|_{\Leb^p}.
\end{equation}
\end{remark}
\begin{proof}
Directly comparing the Fourier--Walsh expansions,
\begin{equation*}
    \Dop_J \Pop_t f = \sum_{S\subseteq[n]:S\supseteq J} \upe^{-t|S|}\widehat{f}(S) \chi_{S\setminus J},\quad \Pop_t \Dop_J f = \sum_{S\subseteq[n]:S\supseteq J} \upe^{-t|S\setminus J|}\widehat{f}(S) \chi_{S\setminus J},
\end{equation*}
and noting that $|S\setminus J| = |S|-|J|$ for $S\supseteq J$, we obtain \eqref{eq:J-derivative-semigroup-decay-Boolean}. Further applying Jensen's inequality (in $p_{t,x}$) to the convex function $(x_J)_{J:|J|=k}\mapsto\sqrt{\sum_{J:|J|=k}|x_J|^2}$ yields \eqref{eq:k-derivative-semigroup-decay-Boolean}.
\end{proof}

Moreover, the noise semigroup satisfies the hypercontractivity:
\begin{proposition}[Bonami--Beckner \cite{Bonami1970Fourier,Beckner1975inequalities}]\label{prop:hypercontractivity-Boolean}
For $f:\cube^n\to \Rbb$, $t\geq 0$ and $p \geq 1+\upe^{-2t}$,
\begin{equation*}
    \left\|\Pop_t f\right\|_{\Leb^2} \leq \left\|f\right\|_{\Leb^{p}}.
\end{equation*}
\end{proposition}

\subsubsection*{Ornstein--Uhlenbeck semigroup in the Gaussian setting}
The OU semigroup is defined via Mehler's formula \cite{Mehler1866}: for $t\geq 0$ and $f\in \Leb^{2}(\Rbb^n,\gamma_n)$,
\begin{equation*}
    \Pop_t f(x) := \Exp_{\zbold\sim \gamma_n} \left[f(\upe^{-t}x+\sqrt{1-\upe^{-2t}}\zbold)\right], \quad x\in \Rbb^n.
\end{equation*}
Analogous to the Boolean setting, $(\Pop_t)_{t\geq 0}$ forms an ergodic symmetric Markov semigroup that is diagonalized by the normalized multivariate Hermite polynomials $\{h_\alpha\}_{\alpha \in \Nbb^n}$:
\begin{equation}\label{eq:semigroup-Fourier-Gaussian}
    \Pop_t f = \sum_{\alpha\in \Nbb^n} \upe^{-t|\alpha|}\widehat{f}(\alpha) h_\alpha.
\end{equation}

The OU semigroup similarly commutes with the normalized high-order derivatives:

\begin{lemma}
\label{lem:derivative-semigroup-decay-Gaussian}
For all $k\in \Nbb$, $f\in \Sob^k(\Rbb^n,\gamma_n)$, $\kappa\in \Nbb^n$ with $|\kappa|=k$, and $t\geq 0$,
\begin{equation}\label{eq:kappa-derivative-semigroup-decay-Gaussian}
    \Dop^\kappa \Pop_t f = \upe^{-t|\kappa|} \Pop_t \Dop^\kappa f,
\end{equation}
and consequently, for all $k\in \Nbb$,
\begin{equation}\label{eq:k-derivative-semigroup-decay-Gaussian}
    \left|\grad^k \Pop_t f\right| \leq \upe^{-kt}\Pop_t \left|\grad^k f\right|.
\end{equation}
\end{lemma}
\begin{remark}
For $p\in [1,\infty]$, we also have the energy-decay estimates:
\begin{equation}\label{eq:derivative-semigroup-decay-Gaussian}
    \left\|\Dop^\kappa \Pop_t f\right\|_{\Leb^p} \leq \upe^{-t|\kappa|} \left\|\Dop_J f\right\|_{\Leb^p}, \quad \left\|\left|\grad^k \Pop_t f\right|\right\|_{\Leb^p} \leq \upe^{-kt} \left\|\left|\grad^k f\right|\right\|_{\Leb^p}.
\end{equation}
\end{remark}
\begin{proof}
Comparing the Fourier--Hermite expansions,
\begin{equation*}
    \Dop^\kappa \Pop_t f = \sum_{\alpha\in \Nbb^n:\alpha\geq \kappa} \sqrt{\binom{\alpha}{\kappa}}\, \upe^{-t|\alpha|} \widehat{f}(\alpha) h_{\alpha-\kappa},
\end{equation*}
\begin{equation*}
    \Pop_t \Dop^\kappa f = \sum_{\alpha\in \Nbb^n:\alpha\geq \kappa} \upe^{-t|\alpha-\kappa|} \sqrt{\binom{\alpha}{\kappa}}\, \widehat{f}(\alpha) h_{\alpha-\kappa},
\end{equation*}
and noting that $|\alpha-\kappa| = |\alpha|-|\kappa|$ for $\alpha\geq \kappa$, we obtain \eqref{eq:kappa-derivative-semigroup-decay-Gaussian}. The second inequality \eqref{eq:k-derivative-semigroup-decay-Gaussian} then follows similarly from Jensen's inequality.
\end{proof}

The OU semigroup also satisfies an identical hypercontractivity:
\begin{proposition}[Nelson--Gross \cite{Nelson1966,Gross1975}]\label{prop:hypercontractivity-Gaussian}
For $f\in \Leb^{2}(\Rbb^n,\gamma_n)$, $t\geq 0$ and $p \geq 1+\upe^{-2t}$,
\begin{equation*}
    \left\|\Pop_t f\right\|_{\Leb^2} \leq \left\|f\right\|_{\Leb^{p}}.
\end{equation*}
\end{proposition}

\section{Structural identities and derivative estimates}\label{subsec:derivative-estimate}

In this section, we establish fundamental structural identities that govern the behavior of both the noise and OU semigroups. These formulas decompose the action of the semigroup applied to the square of a function into a weighted sum of the squared high-order derivatives of the evolved function. From these precise algebraic characterizations, we subsequently deduce Bismut-type derivative estimates and interpolation bounds.

\subsection{Semigroup structural identities}

We first present the structural identities of both the noise semigroup in the Boolean setting and the OU semigroup in the Gaussian setting.

More specifically, the noise semigroup satisfies the following structural identity:
\begin{proposition}
\label{prop:local-derivative-identity-Boolean}
For all $f:\cube^n\to \Rbb$ and $t\geq 0$,
\begin{equation}\label{eq:local-derivative-identity-Boolean}
    \Pop_t \left(\left|f\right|^2\right) = \sum_{J\subseteq [n]} \left(\upe^{2t}-1\right)^{|J|}\left|\Dop_J \Pop_t f\right|^2 = \sum_{k=0}^n \left(\upe^{2t}-1\right)^{k}\left|\grad^k \Pop_t f\right|^2.
\end{equation}
\end{proposition}
\begin{remark}
Combining the structural identity \eqref{eq:local-derivative-identity-Boolean} and the $\Leb^\infty$ contraction property of $\Pop_t$, we immediately deduce the following Bismut-type estimate for high-order gradient tensors:
\begin{equation}\label{eq:Bismut-grad-bound-Boolean}
    \left\|\left|\grad^k \Pop_t f\right|\right\|_{\Leb^\infty} \leq \left(\upe^{2t}-1\right)^{-k/2} \left\|\Pop_t \left|f\right|^2\right\|_{\Leb^\infty}^{1/2} \leq \left(\upe^{2t}-1\right)^{-k/2} \left\|f\right\|_{\Leb^\infty}.
\end{equation}
Although this identity also yields corresponding bounds for individual discrete directional derivatives $\Dop_J$, the direct operator-boundedness of $\Dop_J$ (see Lemma~\ref{lem:derivative-operator-norm}) together with the decay estimate \eqref{eq:derivative-semigroup-decay-Boolean} provides a superior estimate:
\begin{equation}\label{eq:Bismut-J-der-bound-Boolean}
    \left\|\Dop_J \Pop_t f\right\|_{\Leb^\infty} \leq \upe^{-t|J|} \left\|\Dop_J f\right\|_{\Leb^\infty} \leq \upe^{-t|J|} \left\|f\right\|_{\Leb^\infty},
\end{equation}
circumventing the short-time singularity.
\end{remark}

Before establishing Proposition~\ref{prop:local-derivative-identity-Boolean}, we first prove a discrete Bismut-type formula:
\begin{lemma}
\label{lem:BEL-Boolean}
For all $J\subseteq [n]$, $f:\cube^n\to \Rbb$, and $t > 0$,
\begin{equation}\label{eq:BEL-Boolean-identity}
    \Dop_J \Pop_t f(x) = \left(\upe^{2t}-1\right)^{-|J|/2} \Exp_{\ybold\sim p_{t,x}} \left[\phi_J^{t,x}(\ybold)f(\ybold)\right],
\end{equation}
where $\phi_J^{t,x}(y) := \prod_{i\in J} \frac{y_i-\upe^{-t} x_i}{\sqrt{1-\upe^{-2t}}}$.
\end{lemma}

The first-order form ($|J|=1$) of \eqref{eq:BEL-Boolean-identity} was first formally introduced in \cite{IVV2020Rademacher}, serving as a key decoupling tool in the structural analysis of Banach spaces. While the high-order identity here follows directly by iterating the first-order relation over the coordinates in $J$, we provide a short proof for completeness.

\begin{proof}
By the kernel representation \eqref{eq:def-semigroup-Boolean},
\begin{equation*}
    \Dop_J \Pop_t f(x) = \sum_{y\in \cube^n} f(y)\, \Dop_J p_t(\cdot,y) (x),
\end{equation*}
where $p_t(\cdot,y)$ denotes the function $x\mapsto p_t(x,y)$ for a fixed $y$. To compute $\Dop_J p_t(\cdot,y)$, we first consider a single coordinate case $i\in [n]$. By definition of $\Dop_i$ and $p_t$,
\begin{align*}
    \Dop_i p_t(\cdot,y) (x) &= \frac{p_t(x^{(i\mapsto 1)},y)-p_t(x^{(i\mapsto -1)},y)}{2}\\
    &= \frac{\upe^{-t}y_i}{2}\prod_{j\neq i} \frac{1+\upe^{-t}x_j y_j}{2} = \frac{\upe^{-t}y_i}{1+\upe^{-t}x_i y_i}\, p_t(x,y).
\end{align*}
Since $x_i^2 = y_i^2 = 1$, algebraic simplification of the multiplier yields
\begin{equation*}
    \frac{\upe^{-t}y_i}{1+\upe^{-t}x_i y_i} = \frac{\upe^{-t}y_i\left(1-\upe^{-t}x_i y_i\right)}{1-\upe^{-2t}} = \left(\upe^{2t}-1\right)^{-1/2} \frac{y_i-\upe^{-t} x_i}{\sqrt{1-\upe^{-2t}}}.
\end{equation*}
Iterating this relation over $i\in J$, we obtain
\begin{equation*}
    \Dop_J p_t(\cdot,y) (x) = \prod_{i\in J}\frac{\upe^{-t}y_i}{1+\upe^{-t}x_i y_i}\, p_t(x,y) = \left(\upe^{2t}-1\right)^{-|J|/2} \phi_J^{t,x}(y) \, p_t(x,y).
\end{equation*}
Substituting this back into the kernel expansion completes the proof.
\end{proof}

The structural identity then follows from Parseval's identity w.r.t. the local measure:
\begin{proof}[Proof of Proposition \ref{prop:local-derivative-identity-Boolean}]
Since $p_{t,x}$ is exactly the product distribution of independent bits with mean vector $\upe^{-t} x$, an orthonormal basis of $\Leb^2(\cube^n,p_{t,x})$ is given by $\{\phi_J^{t,x}\}_{J\subseteq[n]}$. We can therefore expand $f$ in this local basis as
\begin{equation*}
    f(y) = \sum_{J\subseteq[n]} \widehat{f_{t,x}}(J) \phi_J^{t,x}(y),\quad \widehat{f_{t,x}}(J) := \Exp_{\ybold\sim p_{t,x}} \left[\phi_J^{t,x}(\ybold)f(\ybold)\right].
\end{equation*}
Further by the discrete Bismut-type formula \eqref{eq:BEL-Boolean-identity},
\begin{equation*}
    \widehat{f_{t,x}}(J) = \Exp_{\ybold\sim p_{t,x}} \left[\phi_J^{t,x}(\ybold)f(\ybold)\right] = \left(\upe^{2t}-1\right)^{|J|/2}\Dop_J \Pop_t f(x).
\end{equation*}
Applying Parseval's identity then immediately yields
\begin{equation*}
    \Pop_t \left(\left|f\right|^2\right)(x) = \Exp_{\ybold\sim p_{t,x}} \left[\left|f(\ybold)\right|^2\right] = \sum_{J\subseteq[n]} \widehat{f_{t,x}}(J)^2 = \sum_{J\subseteq[n]} \left(\upe^{2t}-1\right)^{|J|}\left|\Dop_J \Pop_t f(x)\right|^2
\end{equation*}
which completes the proof.
\end{proof}

\begin{remark}
Although derived via local Parseval relations, the identity \eqref{eq:local-derivative-identity-Boolean} is fundamentally a reflection of the intrinsic algebraic and tensor structure of the noise semigroup rather than a particularity of the Boolean setting. To illustrate, on a general product probability space, let $\Eop_i$ be the conditional expectation that integrates out the $i$-th coordinate, and define $\Eop_I := \prod_{i\in I} \Eop_i$ and $\Lop_I := \prod_{i\in I}(\Id-\Eop_i)$. Then the generalized noise semigroup $\Pop_t := \prod_{i\in [n]} (\upe^{-t}\Id + (1-\upe^{-t})\Eop_i)$ satisfies the global pointwise identity
\begin{equation}\label{eq:general-product-identity}
    \Pop_t\left(\left|f\right|^2\right) = \sum_{I,J\subseteq[n]} \left(\upe^t-1\right)^{|I|+|J|}\Eop_I \left(\left|\Lop_I \Lop_J \Pop_t f\right|^2\right).
\end{equation}
An analogous identity also holds for the depolarizing semigroup in quantum settings. This highlights that these formulas universally capture how tensorized Markovian evolutions govern higher-order fluctuations, independent of the underlying geometry.
\end{remark}

In the Gaussian setting, the OU semigroup also satisfies a similar structural identity:
\begin{proposition}\label{prop:local-derivative-identity-Gaussian}
For all $f\in \Leb^2(\Rbb^n,\gamma_n)$ and $t\geq 0$,
\begin{equation}\label{eq:local-derivative-identity-Gaussian}
    \Pop_t \left|f\right|^2 = \sum_{\kappa\in \Nbb^n} \left(\upe^{2t}-1\right)^{|\kappa|}\left|\Dop^\kappa \Pop_t f\right|^2 = \sum_{k=0}^\infty \left(\upe^{2t}-1\right)^{k}\left|\grad^k \Pop_t f\right|^2.
\end{equation}
\end{proposition}
\begin{remark}
The structural identity \eqref{eq:local-derivative-identity-Gaussian}, in conjunction with the $\Leb^\infty$ contraction property of $\Pop_t$, similarly yields the following high-order Bismut-type derivative estimates
\begin{equation}\label{eq:Bismut-grad-bound-Gaussian}
    \left\|\left|\grad^k \Pop_t f\right|\right\|_{\Leb^\infty} \leq \left(\upe^{2t}-1\right)^{-k/2} \left\|f\right\|_{\Leb^\infty}, \ \left\|\Dop^\kappa \Pop_t f \right\|_{\Leb^\infty} \leq \left(\upe^{2t}-1\right)^{-|\kappa|/2} \left\|f\right\|_{\Leb^\infty}.
\end{equation}
Here, the high-order (normalized) partial derivatives are unbounded on $\Leb^\infty(\Rbb^n, \gamma_n)$, thus an analogue of the singularity-free estimate \eqref{eq:Bismut-J-der-bound-Boolean} does not hold for $\Dop^\kappa$.
\end{remark}

Analogously, we first establish a Gaussian Bismut-type formula:
\begin{lemma}
\label{lem:BEL-Gaussian}
For all $\kappa\in \Nbb^n$, $f\in \Leb^2(\Rbb^n,\gamma_n)$, and $t > 0$,
\begin{equation}\label{eq:BEL-Gaussian-identity}
    \Dop^\kappa \Pop_t f(x) = \left(\upe^{2t}-1\right)^{-|\kappa|/2} \Exp_{\ybold\sim p_{t,x}} \left[\psi_\kappa^{t,x}(\ybold)f(\ybold)\right],
\end{equation}
where $p_{t,x}$ denotes the biased multivariate Gaussian distribution with mean vector $\upe^{-t}x$ and covariance matrix $\left(1-\upe^{-2t}\right)\Ibf_n$, and $\psi_\kappa^{t,x}(y) := h_\kappa(\frac{y-\upe^{-t} x}{\sqrt{1-\upe^{-2t}}})$.
\end{lemma}
\begin{proof}
Rewriting Mehler's formula by substituting $y=\upe^{-t}x+\sqrt{1-\upe^{-2t}}z$, we obtain
\begin{equation*}
    \Pop_t f(x) = \int_{\Rbb^n} f(\upe^{-t}x+\sqrt{1-\upe^{-2t}}z)\, \varphi_n(z) \dif z= \int_{\Rbb^n} f(y)\,  p_t(x,y)\dif y
\end{equation*}
where $p_t(x,y):=\left(1-\upe^{-2t}\right)^{-n/2}\varphi_n\left(\frac{y-\upe^{-t}x}{\sqrt{1-\upe^{-2t}}}\right)$ is the density of $p_{t,x}$. For $f\in \Leb^2(\Rbb^n,\gamma_n)$ and $t>0$, the standard regularizing effect of the symmetric Markov semigroup ensures that $\Pop_t f\in \bigcap_{k=1}^\infty \Sob^k(\Rbb^n,\gamma_n)$, which justifies differentiating directly under the integral sign:
\begin{equation*}
    \Dop^\kappa \Pop_t f (x) = \int_{\Rbb^n} f(y)\, \Dop^\kappa p_t(\cdot,y) (x) \dif y.
\end{equation*}
By the chain rule of differentiation,
\begin{equation*}
    \Dop^\kappa p_t(\cdot,y) (x) = \left(1-\upe^{-2t}\right)^{-n/2} \Dop^\kappa \varphi_n\left(\frac{y-\upe^{-t}x}{\sqrt{1-\upe^{-2t}}}\right) \prod_{i\in [n]} \left(\frac{\partial}{\partial x_i} \frac{y_i-\upe^{-t}x_i}{\sqrt{1-\upe^{-2t}}} \right)^{\kappa_i}.
\end{equation*}
The definition of $h_\kappa$ and direct computation yield
\begin{equation*}
    \Dop^\kappa \varphi_n = \left(-1\right)^{|\kappa|}  \varphi_n\, h_{\kappa}, \quad \frac{\partial}{\partial x_i} \frac{y_i-\upe^{-t}x_i}{\sqrt{1-\upe^{-2t}}} = - \left(\upe^{2t}-1\right)^{-1/2}.
\end{equation*}
Altogether, the negative signs $\left(-1\right)^{|\kappa|}$ cancel out perfectly, leaving
\begin{equation*}
    \Dop^\kappa p_t(\cdot,y) (x) = \left(\upe^{2t}-1\right)^{-|\kappa|/2}\, \psi_\kappa^{t,x}(y)\, p_t(x,y).
\end{equation*}
Substituting this back into the kernel expansion completes the proof.
\end{proof}

The structural identity then follows via a similar local Parseval discussion:
\begin{proof}[Proof of Proposition~\ref{prop:local-derivative-identity-Gaussian}]
By construction, the local Hermite polynomials $\{\psi_\kappa^{t,x}\}_{\kappa\in \Nbb^n}$ form an orthonormal basis of $\Leb^2(\Rbb^n,p_{t,x})$, yielding the local Fourier expansion
\begin{equation*}
    f(y) = \sum_{\kappa \in \Nbb^n} \widehat{f_{t,x}}(\kappa) \psi_\kappa^{t,x}(y),\quad \text{where } \widehat{f_{t,x}}(\kappa) := \Exp_{\ybold\sim p_{t,x}} \left[\psi_\kappa^{t,x}(\ybold)f(\ybold)\right].
\end{equation*}
By the Gaussian Bismut-type formula \eqref{eq:BEL-Gaussian-identity}, these local Fourier coefficients satisfy
\begin{equation*}
    \widehat{f_{t,x}}(\kappa) = \left(\upe^{2t}-1\right)^{|\kappa|/2} \Dop^\kappa \Pop_t f(x).
\end{equation*}
Applying Parseval's identity immediately yields
\begin{equation*}
    \Pop_t \left(\left|f\right|^2\right)(x) = \Exp_{\ybold\sim p_{t,x}} \left[\left|f(\ybold)\right|^2\right] = \sum_{\kappa \in \Nbb^n} \widehat{f_{t,x}}(\kappa)^2 = \sum_{\kappa \in \Nbb^n} \left(\upe^{2t}-1\right)^{|\kappa|}\left|\Dop^\kappa \Pop_t f(x)\right|^2,
\end{equation*}
which completes the proof.
\end{proof}

\subsection{Interpolation bounds}

Using the Bismut-type derivative estimates provided by the structural identities, together with the energy-decay estimates provided by the semigroup--derivative commutation laws, we deduce the following interpolation bounds.

We first prove the following semigroup-smoothed interpolation bound for the high-order gradient tensors in the Boolean setting:
\begin{lemma}\label{lem:energy-bound-Boolean}
For all $k\in \Nbb_{\leq n}$, $f:\cube^n\to \Rbb$, $t> 0$ and $p\in [1,2]$,
\begin{equation}\label{eq:energy-bound-Boolean}
    \left\|\left|\grad^k \Pop_t f\right|\right\|_{\Leb^2}^{2} \leq \upe^{-2k\left(p-1\right)t}\left(\upe^{4t}-1\right)^{-k(1-\frac{p}{2})}  \left\|f\right\|_{\Leb^{\infty}}^{2-p} \left\|\left|\grad^k f\right|\right\|_{\Leb^p}^{p}.
\end{equation}
\end{lemma}
\begin{proof}
By definition
\begin{equation*}
    \left\|\left|\grad^k \Pop_t f\right|\right\|_{\Leb^2}^{2} = \sum_{J\subseteq[n]:|J|=k} \langle \Dop_J \Pop_t f, \Dop_J \Pop_{t} f \rangle_{\Leb^2}.
\end{equation*}
By the symmetry of the semigroup $\Pop_t$ and the commutation law \eqref{eq:J-derivative-semigroup-decay-Boolean},
\begin{align*}
    \langle \Dop_J \Pop_t f, \Dop_J \Pop_{t} f \rangle_{\Leb^2} &= \upe^{-t|J|}\langle \Pop_{t} \Dop_J f, \Dop_J \Pop_t f\rangle_{\Leb^2} \\
    &= \upe^{-t|J|}\langle\Dop_J f , \Pop_{t} \Dop_J \Pop_t f\rangle_{\Leb^2} = \langle \Dop_J f, \Dop_J \Pop_{2t} f \rangle_{\Leb^2}.
\end{align*}
Using the Cauchy inequality (w.r.t. summation over $J$) and the H\"{o}lder inequality, we obtain
\begin{equation*}
    \sum_{J\subseteq [n]:|J|=k} \langle \Dop_J f, \Dop_J \Pop_{2t} f \rangle_{\Leb^2} \leq \left\langle\left|\grad^k f\right|,\left|\grad^k \Pop_{2t} f\right|\right\rangle_{\Leb^{2}} \leq \left\|\left|\grad^k f\right|\right\|_{\Leb^{p}} \left\|\left|\grad^k \Pop_{2t} f\right|\right\|_{\Leb^{q}},
\end{equation*}
with $q=\frac{p}{p-1}$. The  logarithmic convexity of $\Leb^{q}$ norms (Lemma~\ref{lem:norm-log-convex}) further yields
\begin{equation*}
    \left\|\left|\grad^k \Pop_{2t} f\right|\right\|_{\Leb^{q}}\leq \left\|\left|\grad^k \Pop_{2t} f\right|\right\|_{\Leb^{2}}^{2\frac{p-1}{p}} \left\|\left|\grad^k \Pop_{2t} f\right|\right\|_{\Leb^{\infty}}^{\frac{2-p}{p}}.
\end{equation*}
For the first term, the energy-decay estimate \eqref{eq:derivative-semigroup-decay-Boolean} yields
\begin{equation*}
    \left\|\left|\grad^k \Pop_{2t} f\right|\right\|_{\Leb^{2}}^2  \leq \upe^{-2kt}\left\|\left|\grad^k \Pop_{t} f\right|\right\|_{\Leb^{2}}^2.
\end{equation*}
For the second term, the Bismut-type gradient estimate \eqref{eq:Bismut-grad-bound-Boolean} yields
\begin{equation*}
    \left\|\left|\grad^k \Pop_{2t} f\right|\right\|_{\Leb^{\infty}} \leq \left(\upe^{4t}-1\right)^{-\frac{k}{2}} \left\|f\right\|_{\Leb^{\infty}}.
\end{equation*}
Combining the bounds, we obtain
\begin{equation*}
    \left\|\left|\grad^k \Pop_t f\right|\right\|_{\Leb^2}^{2} \leq \left(\upe^{-2 kt} \left\|\left|\grad^k \Pop_t f\right|\right\|_{\Leb^2}^{2}\right)^{\frac{p-1}{p}} \left(\left(\upe^{4t}-1\right)^{-\frac{k}{2}} \left\|f\right\|_{\Leb^{\infty}}\right)^{\frac{2-p}{p}} \left\|\left|\grad^k  f\right|\right\|_{\Leb^{p}}.
\end{equation*}
Assume w.l.o.g. that $\left\|\left|\grad^k \Pop_t f\right|\right\|_{\Leb^2}>0$. Raising both sides with power $p$ and cancelling $\left\|\left|\grad^k \Pop_t f\right|\right\|_{\Leb^2}^{2(p-1)}$, we complete the proof.
\end{proof}

Compared to the dimension-free estimates (\eqref{eq:Bismut-grad-bound-Boolean} and \eqref{eq:energy-bound-Boolean}) for the gradient tensors, the discrete derivatives possess superior estimates due to the $\Leb^\infty$ operator-boundedness:
\begin{lemma}\label{lem:derivative-operator-norm}
For $J\subseteq[n]$, $\left\|\Dop_J\right\|_{\Leb^\infty\to \Leb^\infty}=1$.
\end{lemma}
\begin{proof}
Let $i\in [n]$. By definition, for all $f:\cube^n\to \Rbb$ and $x\in \cube^n$ we have
\begin{equation*}
    \left|\Dop_i f(x)\right| = \left|\frac{f(x^{(i\mapsto +1)}) - f(x^{(i\mapsto -1)})}{2}\right| \leq \frac{\left\|f\right\|_{\Leb^\infty} + \left\|f\right\|_{\Leb^\infty}}{2} = \left\|f\right\|_{\infty},
\end{equation*}
and hence $\left\|\Dop_i\right\|_{\Leb^\infty\to \Leb^\infty}\leq 1$. Iterating across $i\in J$ gives $\left\|\Dop_J\right\|_{\Leb^\infty\to \Leb^\infty}\leq 1$. Conversely, the equality is attained since $\left\|\Dop_J \chi_J\right\|_{\Leb^\infty}=1= \left\|\chi_J\right\|_{\Leb^\infty}$.
\end{proof}
\begin{remark}
Using the operator norm bound, for all $p\in [1,2]$ we have
\begin{equation*}
    \left\|\Dop_J f\right\|_{\Leb^2}^{2} \leq \left\|\Dop_J f\right\|_{\Leb^\infty}^{2-p} \left\|\Dop_J f\right\|_{\Leb^p}^{p} \leq \left\|f\right\|_{\Leb^\infty}^{2-p} \left\|\Dop_J f\right\|_{\Leb^p}^{p}.
\end{equation*}
Combining the energy-decay estimate \eqref{eq:derivative-semigroup-decay-Boolean} yields a sharp interpolation bound
\begin{equation*}
    \left\|\Dop_J \Pop_t f\right\|_{\Leb^2}^{2} \leq \upe^{-2t|J|} \left\|\Dop_J f\right\|_{\Leb^2}^{2} \leq \upe^{-2t|J|} \left\|f\right\|_{\Leb^\infty}^{2-p} \left\|\Dop_J f\right\|_{\Leb^p}^{p}
\end{equation*}
circumventing the short-time singularity in \eqref{eq:energy-bound-Boolean}.
\end{remark}

In the Gaussian setting, analogous to Lemma~\ref{lem:energy-bound-Boolean}, the Bismut-type derivative bounds \eqref{eq:Bismut-grad-bound-Gaussian} combined with the energy-decay estimates \eqref{eq:derivative-semigroup-decay-Gaussian} also yield the following bounds:
\begin{lemma}\label{lem:energy-bound-Gaussian}
For all $k\in \Nbb$, $f\in \Sob^k(\Rbb^n,\gamma_n)\cap \Leb^{\infty}(\Rbb^n,\gamma_n)$, $t> 0$ and $p\in [1,2]$,
\begin{equation}\label{eq:energy-bound-Gaussian}
    \left\|\left|\grad^k \Pop_t f\right|\right\|_{\Leb^2}^{2} \leq \upe^{-2k\left(p-1\right)t}\left(\upe^{4t}-1\right)^{-k(1-\frac{p}{2})}  \left\|f\right\|_{\Leb^{\infty}}^{2-p} \left\|\left|\grad^k f\right|\right\|_{\Leb^p}^{p},
\end{equation}
and further for all $\kappa\in \Nbb^n$ with $|\kappa|=k$,
\begin{equation}\label{eq:partial-energy-bound-Gaussian}
    \left\|\Dop^\kappa \Pop_t f\right\|_{\Leb^2}^{2} \leq \upe^{-2k\left(p-1\right)t}\left(\upe^{4t}-1\right)^{-k(1-\frac{p}{2})}  \left\|f\right\|_{\Leb^{\infty}}^{2-p} \left\|\Dop^\kappa f\right\|_{\Leb^p}^{p}.
\end{equation}
\end{lemma}

\section{High-order Besov-type variance functionals}\label{sec:high-order-variance}

In this section, we introduce and analyze the high-order Besov-type variance functionals, which serve as the engine of our high-order variance-decay framework. By construction, these functionals are intrinsically tied to the Besov $\Besov^0_{2,2/k}$ quasi-norms, thereby capturing intermediate regularity scales that lie strictly between classical Sobolev $\Sob^k$ and $\Leb^2$ spaces.

More specifically, in the Boolean setting, for $f:\cube^n\to \Rbb$ and nonzero $k\in \Nbb_{\leq n}$, we define the \textit{$k$-th order variance functional} of $f$ by
\begin{equation}\label{eq:k-var-def-Boolean}
    \Vop^k(f) := \left(\int_0^\infty 2 \left\|\left|\grad^k \Pop_t f\right|\right\|_{\Leb^2}^{2/k} \dif t\right)^{k},
\end{equation}
and further for nonempty $J\subseteq [n]$, we define the \textit{$J$-partial variance functional} of $f$ by
\begin{equation}\label{eq:J-var-def-Boolean}
    \Vop_J(f) := \left(\int_0^\infty 2 \left\|\Dop_J \Pop_t f\right\|_{\Leb^2}^{2/|J|} \dif t\right)^{|J|}.
\end{equation}
For the trivial cases $k=0$ and $J=\varnothing$, we naturally set $\Vop^0(f)=\Vop_\varnothing(f):=\left\|f\right\|_{\Leb^2}^2$.

The definition in the Gaussian setting is similar: for $f\in \Leb^2(\Rbb^n,\gamma_n)$ and nonzero $k\in \Nbb$, we also define the \textit{$k$-th order variance functional} of $f$ by
\begin{equation}\label{eq:k-var-def-Gaussian}
    \Vop^k(f) := \left(\int_0^\infty 2 \left\|\left|\grad^k \Pop_t f\right|\right\|_{\Leb^2}^{2/k} \dif t\right)^{k},
\end{equation}
and further for nonzero $\kappa\in \Nbb^n$, we define the \textit{$\kappa$-partial variance functional} of $f$ by
\begin{equation}\label{eq:kappa-var-def-Gaussian}
    \Vop^{\kappa}(f) := \left(\int_0^\infty 2 \left\|\Dop^\kappa \Pop_t f\right\|_{\Leb^2}^{2/|\kappa|} \dif t\right)^{|\kappa|}.
\end{equation}
For the trivial cases $k=0$ and $\kappa=\ovec$, we similarly set $\Vop^0(f)=\Vop^{\ovec}(f):=\left\|f\right\|_{\Leb^2}^2$. Note that by the standard smoothing property $\Pop_t (\Leb^2(\Rbb^n,\gamma_n))\subseteq \bigcap_{k=1}^\infty \Sob^k(\Rbb^n,\gamma_n)$ for $t>0$, $\Vop^k(f)$ and $\Vop^{\kappa}(f)$ are well-defined here for all $f\in \Leb^2(\Rbb^n,\gamma_n)$, although they may not be finite. We further write $\Vsp^k(\Rbb^n,\gamma_n):=\left\{f\in \Leb^2(\Rbb^n,\gamma_n): \Vop^k(f)<\infty\right\}$.

\begin{remark}
As mentioned in Section~\ref{subsec:outline}, for $k\geq 1$, the high-order variance functional $\Vop^k(f)$ behaves as the squared equivalent Besov $\Besov_{2,2/k}^0$ quasi-seminorm and recovers $\Var(f)$ when $k=1$. To illustrate this Besov nature, we first recall Stein's semigroup characterization of Besov scales \cite{Stein1970topics} (see also \cite{Coulhon2012heat,Cao2022heat} for modern formalizations). An equivalent quasi-norm of the Besov space $\Besov_{p,q}^s$ with regularity parameter $s\in \Rbb$, integrability parameter $p\in (0,\infty]$, and fine-tuning parameter $q\in (0,\infty]$ is given by
\begin{equation*}
    \left\|f\right\|_{\Besov_{p,q}^s} := \left|\Eop f\right| + \left( \int_0^\infty \left( t^{m - \frac{s}{2}} \left\|(-\Lop)^m \Pop_t f\right\|_{\Leb^p} \right)^q \frac{\dif t}{t} \right)^{1/q},
\end{equation*}
where $\Pop_t=\upe^{t\Lop}$ is the associated semigroup of the underlying space, and $m$ is an arbitrarily chosen number satisfying $2m>s$. In both the Boolean and Gaussian settings, we have
\begin{equation*}
    \left\|\left|\grad^k f\right|\right\|_{\Leb^2}^2 = \sum_{m\geq k} \binom{m}{k} \Wop^m(f) \asymp_k \sum_{m\geq k} m^k \Wop^m(f) = \left\|(-\Lop)^{k/2} f^{\geq k}\right\|_{\Leb^2}^2.
\end{equation*}
Hence, matching the parameters $s=0$, $p=2$, and $q=2/k$, we obtain
\begin{equation*}
    \Vop^k(f) \asymp_k \left(\int_0^\infty \left(t^{k/2}\left\|(-\Lop)^{k/2} \Pop_t f^{\geq k}\right\|_{\Leb^2}\right)^{2/k} \frac{\dif t}{t}\right)^{k} = \left\|f^{\geq k}\right\|_{\Besov_{2,2/k}^0}^2,
\end{equation*}
yielding exactly the structural equivalence $\Vsp^k=\Besov_{2,2/k}^0$.
\end{remark}

\subsection{Comparison inequalities and space embeddings}\label{subsec:high-order-variance-compare}

The following results demonstrate that these high-order variance functionals lie exactly between the high-order energies and Fourier weights.

More specifically, in the Boolean setting we have:
\begin{proposition}\label{prop:high-order-var-bound-Boolean}
For all $k\in \Nbb_{\leq n}$ and $f:\cube^n\to \Rbb$,
\begin{equation}\label{eq:high-order-var-bound-Boolean}
    \Wop^{\geq k}(f) \leq \sum_{J\subseteq [n]:|J|=k} \Vop_J(f) \leq \Vop^k(f) \leq \left\|\left|\grad^k f\right|\right\|_{\Leb^2}^2.
\end{equation}
\end{proposition}
\begin{proof}
When $k=0$, the desired inequality holds trivially:
\begin{equation*}
    \Wop^{\geq 0}(f) = \Vop_\varnothing(f) = \Vop^0(f) =  \left\|\left|\grad^0 f\right|\right\|_{\Leb^2}^2 =\left\|f\right\|_{\Leb^2}^2.
\end{equation*}
Assume $k\geq 1$. The energy-decay estimate \eqref{eq:derivative-semigroup-decay-Boolean} yields
\begin{equation*}
    \Vop^k(f) = \left(\int_0^\infty 2 \left\|\left|\grad^k \Pop_t f\right|\right\|_{\Leb^2}^{2/k} \dif t\right)^{k} \leq \left(\int_0^\infty 2\, \upe^{- 2t} \left\|\left|\grad^k f\right|\right\|_{\Leb^2}^{2/k} \dif t\right)^{k} = \left\|\left|\grad^k f\right|\right\|_{\Leb^2}^{2},
\end{equation*}
which proves the energy upper bound. On the other hand, decomposing
\begin{equation*}
    \left\|\left|\grad^k \Pop_t f\right|\right\|_{\Leb^2}^{2} = \sum_{J\subseteq [n]:|J|=k}  \left\|\Dop_J \Pop_t f\right\|_{\Leb^2}^{2},
\end{equation*}
the reverse Minkowski inequality of the $\Leb^{1/k}$ quasi-norm (w.r.t. integration in $t$) then yields
\begin{align*}
    \Vop^k(f) &= \left(\int_0^\infty 2 \left\|\left|\grad^k \Pop_t f\right|\right\|_{\Leb^2}^{2/k} \dif t\right)^{k}\\
    &\geq \sum_{J\subseteq [n]:|J|=k} \left(\int_0^\infty 2 \left\|\Dop_J \Pop_t f\right\|_{\Leb^2}^{2/k} \dif t\right)^{k}  = \sum_{J\subseteq [n]:|J|=k} \Vop_J(f).
\end{align*}
Further invoking the Fourier--Walsh expansion
\begin{equation*}
    \left\|\Dop_J \Pop_t f\right\|_{\Leb^2}^{2} = \sum_{S\subseteq [n]:S\supseteq J} \upe^{-2t|S|}\widehat{f}(S)^2,
\end{equation*}
the reverse Minkowski inequality similarly yields
\begin{align*}
    \Vop_J(f) &= \left(\int_0^\infty 2 \left\|\Dop_J \Pop_t f\right\|_{\Leb^2}^{2/|J|} \dif t\right)^{|J|} \\
    &\geq \sum_{S\subseteq [n]:S\supseteq J} \left(\int_0^\infty 2\, \upe^{-2t|S|/|J|} \dif t\right)^{|J|} \widehat{f}(S)^2 = \sum_{S\subseteq [n]:S\supseteq J} \left(\frac{|J|}{|S|}\right)^{|J|} \widehat{f}(S)^2.
\end{align*}
Summing over all subsets $J\subseteq [n]$ with $|J|=k$, we obtain
\begin{align*}
    \sum_{J\subseteq [n]:|J|=k} \Vop_J(f) &\geq \sum_{J\subseteq [n]:|J|=k} \sum_{S\subseteq [n]:S\supseteq J} \left(\frac{|J|}{|S|}\right)^{|J|} \widehat{f}(S)^2\\
    &= \sum_{m=k}^n \left(\frac{k}{m}\right)^{k} \sum_{S\subseteq [n]:|S|=m} \sum_{J\subseteq S:|J|=k}  \widehat{f}(S)^2\\
    &= \sum_{m=k}^n \left(\frac{k}{m}\right)^{k} \binom{m}{k}\sum_{S\subseteq [n]:|S|=m} \widehat{f}(S)^2 \geq  \Wop^{\geq k}(f).
\end{align*}
where we used the standard bound $\binom{m}{k} \geq \left(\frac{m}{k}\right)^{k}$ for $m\geq k$. This completes the proof.
\end{proof}

In the Gaussian setting, we similarly have:
\begin{proposition}\label{prop:high-order-var-bound-Gaussian}
For all $k\in \Nbb$ and $f\in \Sob^k(\Rbb^n,\gamma_n)$,
\begin{equation}\label{eq:high-order-var-bound-Gaussian}
    \Wop^{\geq k}(f) \leq \sum_{\kappa\in \Nbb^n:|\kappa|=k} \Vop^\kappa(f)  \leq \Vop^k(f) \leq \left\|\left|\grad^k f\right|\right\|_{\Leb^2}^2.
\end{equation}
\end{proposition}
\begin{proof}
When $k=0$, the desired inequality is also trivial:
\begin{equation*}
    \Wop^{\geq 0}(f) = \Vop^{\ovec}(f) = \Vop^0(f) =  \left\|\left|\grad^0 f\right|\right\|_{\Leb^2}^2 =\left\|f\right\|_{\Leb^2}^2.
\end{equation*}
Let $k\geq 1$. The energy upper bound follows analogously from the energy-decay estimate \eqref{eq:derivative-semigroup-decay-Gaussian}. For the lower bound, the reverse Minkowski inequality together with the expansions
\begin{equation*}
    \left\|\left|\grad^k \Pop_t f\right|\right\|_{\Leb^2}^{2} = \sum_{\kappa\in \Nbb^n:|\kappa|=k} \left\|\Dop^\kappa \Pop_t f\right\|_{\Leb^2}^{2}, \ \  \left\|\Dop^\kappa \Pop_t f\right\|_{\Leb^2}^{2} = \sum_{\alpha\in \Nbb^n:\alpha\geq \kappa} \upe^{-2t|\alpha|}\binom{\alpha}{\kappa}  \widehat{f}(\alpha)^2,
\end{equation*}
similarly yields
\begin{equation*}
    \Vop^k(f) \geq \sum_{\kappa\in \Nbb^n:|\kappa|=k} \left(\int_0^\infty 2 \left\|\Dop^\kappa \Pop_t f\right\|_{\Leb^2}^{2/k} \dif t\right)^{k}  = \sum_{\kappa\in \Nbb^n:|\kappa|=k} \Vop^\kappa(f),
\end{equation*}
and further
\begin{align*}
    \sum_{\kappa\in \Nbb^n:|\kappa|=k} \Vop^\kappa(f) &\geq \sum_{\kappa\in \Nbb^n:|\kappa|=k}\sum_{\alpha\in \Nbb^n:\alpha\geq \kappa} \left(\int_0^\infty 2\, \upe^{-2t|\alpha|/|\kappa|} \dif t\right)^{|\kappa|} \binom{\alpha}{\kappa}  \widehat{f}(\alpha)^2\\
    &\geq \sum_{\kappa\in \Nbb^n:|\kappa|=k}\sum_{\alpha\in \Nbb^n:\alpha\geq \kappa} \left(\frac{|\kappa|}{|\alpha|}\right)^{|\kappa|} \binom{\alpha}{\kappa} \widehat{f}(\alpha)^2\\
    &= \sum_{m=k}^\infty \left(\frac{k}{m}\right)^{k} \sum_{\alpha\in \Nbb^n:|\alpha|=m} \sum_{\kappa\leq \alpha:|\kappa|=k}  \binom{\alpha}{\kappa} \widehat{f}(\alpha)^2\\
    &= \sum_{m=k}^\infty \left(\frac{k}{m}\right)^{k} \binom{m}{k} \sum_{\alpha\in \Nbb^n:|\alpha|=m} \widehat{f}(\alpha)^2 \geq \Wop^{\geq k}(f),
\end{align*}
where we invoked Vandermonde's convolution identity $\sum_{\kappa\leq \alpha:|\kappa|=k}  \binom{\alpha}{\kappa} = \binom{m}{k}$ for $|\alpha|=m$, alongside the standard bound $\binom{m}{k} \geq \left(\frac{m}{k}\right)^{k}$ for $m\geq k$. This completes the proof.
\end{proof}

\begin{remark}
While the Besov-type nature of the high-order variance functionals and the space embedding $\Sob^k\subseteq \Besov^0_{2,2/k} = \Vsp^k \subseteq \Leb^2$ already provide a qualitative picture
\begin{equation*}
    \Wop^{\geq k}(f) = \left\|f^{\geq k}\right\|_{\Leb^2}^2 \lesssim_k \left\|f^{\geq k}\right\|_{\Besov^0_{2,2/k}}^2 \asymp_k \Vop^k(f) \lesssim_k \left\|f^{\geq k}\right\|_{\Sob^k}^2 \asymp_k \left\|\left|\grad^k f\right|\right\|_{\Leb^2}^2,
\end{equation*}
Propositions~\ref{prop:high-order-var-bound-Boolean} and \ref{prop:high-order-var-bound-Gaussian} strengthen this relation into precise, non-asymptotic inequalities and reveal exactly how the total high-order variances decompose into partial ones.
\end{remark}

\subsection{Dynamical properties under semigroup evolution}\label{subsec:high-order-var-decay}

Next, we establish characterizations for the dynamical behaviors of the high-order variance functionals under the semigroup evolution.

As a standard result, the canonical variance possesses logarithmic convexity along the trajectories of symmetric Markov semigroups (see, e.g., \cite[Lemma 4.2.6]{BGL2014Markov}). The following results demonstrate that high-order extensions naturally inherit this structural property.

More specifically, in the Boolean setting, we have:
\begin{proposition}\label{prop:var-log-convex-Boolean}
Let $k\in \Nbb_{\leq n}$, $J\subseteq [n]$, and $f:\cube^n\to \Rbb$. Then both the mappings $t\mapsto \Vop^k(\Pop_t f)$ and $t\mapsto \Vop_J(\Pop_t f)$ are logarithmically convex on $[0,+\infty)$.
\end{proposition}
\begin{proof}
First, the mapping $t\mapsto \left\|\left|\grad^k \Pop_t f\right|\right\|_{\Leb^2}^2$ is logarithmically convex on $[0,+\infty)$, as it is a positive linear combination of exponential functions:
\begin{equation*}
    \left\|\left|\grad^k \Pop_t f\right|\right\|_{\Leb^2}^2 = \sum_{m=k}^n \upe^{-2mt} \binom{m}{k} \Wop^m(f).
\end{equation*}
When $k=0$, the assertion is trivial since $\Vop^0(\Pop_t f)=\left\|\Pop_t f\right\|_{\Leb^2}^2 = \left\|\left|\grad^0 \Pop_t f\right|\right\|_{\Leb^2}^2$. For $k\geq 1$, let $\lambda\in [0,1]$ and $t_1,t_2\geq 0$, and write $t_\lambda := \lambda t_1 + (1-\lambda)t_2$. Then the logarithmic convexity of $t\mapsto \left\|\left|\grad^k \Pop_t f\right|\right\|_{\Leb^2}^2$ combined with the H\"older inequality yields
\begin{align*}
    \Vop^k(\Pop_{t_\lambda} f) &= \left(\int_0^\infty 2 \left\|\left|\grad^k \Pop_{t+ t_\lambda} f\right|\right\|_{\Leb^2}^{2/k} \dif t\right)^{k}\\
    &\leq \left(\int_0^\infty 2 \left(\left\|\left|\grad^k \Pop_{t+t_1} f\right|\right\|_{\Leb^2}^{2/k}\right)^\lambda  \left(\left\|\left|\grad^k \Pop_{t+t_2} f\right|\right\|_{\Leb^2}^{2/k} \right)^{1-\lambda}\dif t \right)^{k}\\
    &\leq \left(\int_0^\infty 2 \left\|\left|\grad^k \Pop_{t+t_1} f\right|\right\|_{\Leb^2}^{2/k}\dif t \right)^{k \lambda} \left(\int_0^\infty 2 \left\|\left|\grad^k \Pop_{t+t_2} f\right|\right\|_{\Leb^2}^{2/k} \dif t \right)^{k\left(1-\lambda\right)}\\
    &= \Vop^k(\Pop_{t_1} f)^\lambda \Vop^k(\Pop_{t_2} f)^{1-\lambda},
\end{align*}
which proves the logarithmic convexity of $t\mapsto \Vop^k(\Pop_t f)$. Similarly, by expressing
\begin{equation*}
    \left\|\Dop_J \Pop_t f\right\|_{\Leb^2}^2 = \sum_{S\subseteq [n]:S\supseteq J} \upe^{-2t|S|}\widehat{f}(S)^2,
\end{equation*}
the mapping $t\mapsto \left\|\Dop_J \Pop_t f\right\|_{\Leb^2}^2$ is also logarithmically convex, and an identical argument yields the logarithmic convexity of $t\mapsto \Vop_J(\Pop_t f)$.
\end{proof}

In the Gaussian setting, we similarly have:
\begin{proposition}\label{prop:var-log-convex-Gaussian}
Let $k\in \Nbb$, $\kappa\in \Nbb^n$ with $|\kappa|=k$, and $f\in\Vsp^k(\Rbb^n,\gamma_n)$. Then both the mappings $t\mapsto \Vop^k(\Pop_t f)$ and $t\mapsto \Vop^\kappa(\Pop_t f)$ are logarithmically convex on $[0,+\infty)$.
\end{proposition}
\begin{proof}
Similarly expressing
\begin{equation*}
    \left\|\left|\grad^k \Pop_t f\right|\right\|_{\Leb^2}^2 = \sum_{m= k}^\infty \upe^{-2mt} \binom{m}{k} \Wop^m(f),\ \ \left\|\Dop^\kappa \Pop_t f\right\|_{\Leb^2}^2 = \sum_{\alpha\in \Nbb^n:\alpha\geq \kappa} \upe^{-2t|\alpha|} \binom{\alpha}{\kappa} \widehat{f}(\alpha)^2,
\end{equation*}
as positive linear combinations of exponential functions, the mappings $t\mapsto \left\|\left|\grad^k \Pop_t f\right|\right\|_{\Leb^2}^2$ and $t\mapsto \left\|\Dop^\kappa \Pop_t f\right\|_{\Leb^2}^2$ are both logarithmically convex, yielding the logarithmic convexity of $t\mapsto \Vop^k(\Pop_t f)$ and $t\mapsto \Vop^\kappa(\Pop_t f)$ in an identical manner to Proposition~\ref{prop:var-log-convex-Boolean}.
\end{proof}

\begin{remark}
In particular, the logarithmic convexity of these high-order variance functionals yields a useful endpoint localization effect: a short-time decay estimate of the form \eqref{eq:high-order-var-decay-abstract} holds for all $t\in [0,\epsilon]$ if and only if it is satisfied at the endpoint $t=\epsilon$.
\end{remark}

We conclude this section by deriving sharp decay estimates of these high-order variance functionals along the trajectories via hypercontractivity.

In the Boolean setting, we prove the following:
\begin{lemma}\label{lem:var-decay-Boolean}
For all $k\in \Nbb_{\leq n}$, $f:\cube^n\to \Rbb$, $q\in [1,2)$ and $t\geq 0$,
\begin{equation}\label{eq:var-decay-Boolean}
    \Vop^k(\Pop_t f) \leq \upe^{-2kt} \Vop^k(f)^{1-\vartheta} \min\left\{\Vop^k(f),\,\left\|\left|\grad^k f\right|\right\|_{\Leb^q}^{2},\, \sum_{J\subseteq[n]:|J|=k}\left\|\Dop_J f\right\|_{\Leb^q}^2\right\}^{\vartheta},
\end{equation}
and similarly, for all $J\subseteq [n]$,
\begin{equation}\label{eq:partial-var-decay-Boolean}
    \Vop_J(\Pop_t f) \leq \upe^{-2t |J|} \Vop_J(f)^{1-\vartheta} \min\left\{\Vop_J(f),\,\left\|\Dop_J f\right\|_{\Leb^q}^2\right\}^{\vartheta},
\end{equation}
where $\vartheta=\vartheta(q,t) \triangleq \min\left\{\frac{q}{2-q}\tanh(t), 1\right\}$.
\end{lemma}

\begin{proof}
Assume w.l.o.g. $k\geq 1$. By the definition of $\Vop^k$ and the bound \eqref{eq:k-derivative-semigroup-decay-Boolean} in Lemma~\ref{lem:derivative-semigroup-decay-Boolean},
\begin{equation*}
    \Vop^k(\Pop_t f)^{1/k} = \int_0^\infty 2 \left\|\left|\grad^k \Pop_{t+s} f\right|\right\|_{\Leb^2}^{2/k} \dif s \leq \upe^{-2t}\int_0^\infty 2 \left\|\Pop_t\left|\grad^k \Pop_{s} f\right|\right\|_{\Leb^2}^{2/k} \dif s.
\end{equation*}
Hypercontractivity (Proposition \ref{prop:hypercontractivity-Boolean}) combined with the logarithmic convexity of $\Leb^p$ norms (Lemma~\ref{lem:norm-log-convex})  then yields
\begin{equation*}
    \left\|\Pop_t\left|\grad^k \Pop_{s} f\right|\right\|_{\Leb^2} \leq \left\|\left|\grad^k \Pop_{s} f\right|\right\|_{\Leb^{p}} \leq \left\| \left|\grad^k \Pop_{s} f\right|\right\|_{\Leb^{q}}^{\vartheta} \left\|\left|\grad^k \Pop_{s} f\right|\right\|_{\Leb^{2}}^{1-\vartheta}
\end{equation*}
with $p=1+\upe^{-2t}$ and $\vartheta=\min\left\{\frac{1/p-1/2}{1/q-1/2},1\right\}$. By the H\"older inequality,
\begin{align*}
    &\int_0^\infty 2 \left\|\left|\grad^k \Pop_{s} f\right|\right\|_{\Leb^{q}}^{2\vartheta/k} \left\| \left|\grad^k \Pop_{s} f\right|\right\|_{\Leb^{2}}^{2(1-\vartheta)/k} \dif s \\
    &\hspace{2cm}\leq \left(\int_0^\infty 2 \left\|\left|\grad^k \Pop_{s} f\right|\right\|_{\Leb^2}^{2/k} \dif s \right)^{1-\vartheta} \left(\int_0^\infty 2 \left\|\left|\grad^k \Pop_{s} f\right|\right\|_{\Leb^q}^{2/k} \dif s\right)^{\vartheta}.
\end{align*}
The first term recovers $\Vop^k$:
\begin{equation*}
    \int_0^\infty 2 \left\|\left|\grad^k \Pop_{s} f\right|\right\|_{\Leb^2}^{2/k} \dif s = \Vop^k(f)^{1/k}.
\end{equation*}
The second term can be bounded by $\Vop^k$ via the monotonicity of $\Leb^q$ norms:
\begin{equation*}
    \int_0^\infty 2 \left\|\left|\grad^k \Pop_{s} f\right|\right\|_{\Leb^q}^{2/k} \dif s\leq\int_0^\infty 2 \left\|\left|\grad^k \Pop_{s} f\right|\right\|_{\Leb^2}^{2/k} \dif s = \Vop^k(f)^{1/k},
\end{equation*}
and also via the energy-decay estimate \eqref{eq:derivative-semigroup-decay-Boolean} due to Lemma~\ref{lem:derivative-semigroup-decay-Boolean}:
\begin{equation*}
    \int_0^\infty 2 \left\|\left|\grad^k \Pop_{s} f\right|\right\|_{\Leb^q}^{2/k} \dif s \leq \left\|\left|\grad^k f\right|\right\|_{\Leb^q}^{2/k}  \int_0^\infty 2\, \upe^{-2s} \dif s = \left\|\left|\grad^k f\right|\right\|_{\Leb^q}^{2/k}.
\end{equation*}
Combining these estimates, we obtain
\begin{equation*}
    \Vop^k(\Pop_t f) \leq \upe^{-2kt}\Vop^k(f)^{1-\vartheta} \min\left\{\Vop^k(f),\,\left\|\left|\grad^k f\right|\right\|_{\Leb^q}^{2}\right\}^{\vartheta}.
\end{equation*}
Applying the same line of reasoning to the alternative expansion
\begin{equation*}
    \Vop^k(\Pop_t f)^{1/k} =\int_0^\infty 2 \left(\sum_{J\subseteq[n]:|J|=k} \left\|\Dop_J \Pop_{t+s} f\right\|_{\Leb^2}^{2}\right)^{1/k} \dif s,
\end{equation*}
we similarly deduce that
\begin{equation*}
    \Vop^k(\Pop_t f) \leq \upe^{-2kt}\Vop^k(f)^{1-\vartheta} \min\left\{\Vop^k(f),\, \sum_{J\subseteq[n]:|J|=k}\left\|\Dop_J f\right\|_{\Leb^q}^2\right\}^{\vartheta}.
\end{equation*}
Taking the minimum of these two upper bounds completes the proof of \eqref{eq:var-decay-Boolean}.

Similarly, assuming w.l.o.g. $|J|\geq 1$, the definition of $\Vop_J$ yields
\begin{equation*}
    \Vop_J(\Pop_t f)^{1/|J|} = \int_0^\infty 2 \left\|\Dop_J \Pop_{t+s} f\right\|_{\Leb^2}^{2/|J|} \dif s.
\end{equation*}
Following the same procedure establishes \eqref{eq:partial-var-decay-Boolean}.
\end{proof}

In the Gaussian setting, invoking the Gaussian hypercontractivity (Proposition~\ref{prop:hypercontractivity-Gaussian}) and the corresponding semigroup--derivative commutation law (Lemma~\ref{lem:derivative-semigroup-decay-Gaussian}), we obtain the following parallel estimates.

\begin{lemma}\label{lem:var-decay-Gaussian}
For all $k\in \Nbb$, $f\in \Sob^k(\Rbb^n,\gamma_n)$, $q\in [1,2)$ and $t\geq 0$,
\begin{equation}\label{eq:var-decay-Gaussian}
    \Vop^k(\Pop_t f) \leq \upe^{-2kt} \Vop^k(f)^{1-\vartheta} \min\left\{\Vop^k(f),\,\left\|\left|\grad^k f\right|\right\|_{\Leb^q}^{2},\, \sum_{\kappa\in \Nbb^n:|\kappa|=k}\left\|\Dop^\kappa f\right\|_{\Leb^q}^2\right\}^{\vartheta},
\end{equation}
and further for all $\kappa\in \Nbb^n$ with $|\kappa|=k$,
\begin{equation}\label{eq:partial-var-decay-Gaussian}
    \Vop^\kappa(\Pop_t f) \leq \upe^{-2kt} \Vop^\kappa(f)^{1-\vartheta} \min\left\{\Vop^\kappa(f),\,\left\|\Dop^\kappa f\right\|_{\Leb^q}^2\right\}^{\vartheta},
\end{equation}
where $\vartheta=\vartheta(q,t) := \min\left\{\frac{q}{2-q}\tanh(t), 1\right\}$.
\end{lemma}

\section{Proofs of main theorems}

In this section, we give the proofs of the main theorems stated in Section~\ref{subsec:results}.

\subsection{The master high-order energy--variance inequalities}\label{subsec:high-order-energy-var-proof}

By combining the interpolation bounds and high-order variance-decay estimates developed in the preceding sections, we first establish the following core inequalities relating high-order (total or partial) energies to their corresponding variance functionals. These master inequalities subsequently yield our main results as direct corollaries.

More specifically, by leveraging Lemmas \ref{lem:energy-bound-Boolean} and \ref{lem:var-decay-Boolean}, we establish the following total high-order energy--variance inequality in the Boolean setting:
\begin{lemma}\label{lem:high-order-energy-var-Boolean}
For all $k\in \Nbb_{\leq n}$, $f:\cube^n\to \Rbb$, and $1\leq q\leq p\leq 2$,
\begin{align*}
    &\left\|f\right\|_{\Leb^\infty}^{2-p} \left\|\left|\grad^k f\right|\right\|_{\Leb^p}^{p} \geq M(p)^k\Vop^{k}(f) \\
    &\hspace{1cm} \cdot\left[1+ \frac{q}{2k(2-q)}\max\left\{\ln^+\left(\frac{\Vop^k(f)}{\left\|\left|\grad^k f\right|\right\|_{\Leb^q}^2}\right),\, \ln^+\left(\frac{\Vop^k(f)}{\sum_{J:|J|=k} \left\|\Dop_J f\right\|_{\Leb^q}^2}\right)\right\}\right]^\frac{kp}{2},
\end{align*}
where $M(p) = p \, 2^{-\frac{p}{2}} \max_{x\geq 0} \frac{1-\upe^{-x}}{x^{p/2}}\gtrsim 1$.
\end{lemma}
\begin{proof}
Assume w.l.o.g. $k\geq 1$. Denote $\epsilon_0(q):=\arctanh(\frac{2-q}{q})$ and
\begin{equation*}
    \Rscr(f,q) := 1+ \frac{q}{2k(2-q)}\max\left\{\ln^+\left(\frac{\Vop^k(f)}{\left\|\left|\grad^k f\right|\right\|_{\Leb^q}^2}\right),\, \ln^+\left(\frac{\Vop^k(f)}{\sum_{J:|J|=k} \left\|\Dop_J f\right\|_{\Leb^q}^2}\right)\right\}.
\end{equation*}
Let $\epsilon\in (0,\epsilon_0(q)]$. Then $\frac{q}{2-q}\tanh(\epsilon)\leq 1$, and hence by \eqref{eq:var-decay-Boolean} in Lemma~ \ref{lem:var-decay-Boolean}, we have
\begin{equation*}
    \Vop^k(\Pop_\epsilon f) \leq \upe^{-2k\epsilon} \Vop^k(f) \min\left\{1,\,\frac{\left\|\left|\grad^k f\right|\right\|_{\Leb^q}^{2}}{\Vop^k(f)},\, \frac{\sum_{J:|J|=k}\left\|\Dop_J f\right\|_{\Leb^q}^2}{\Vop^k(f)}\right\}^{\frac{q}{2-q}\tanh(\epsilon)}.
\end{equation*}
Since $\epsilon\geq \tanh(\epsilon)$, we obtain
\begin{align*}
    \Vop^k(\Pop_\epsilon f) &\leq \upe^{-2k\tanh(\epsilon)} \Vop^k(f) \min\left\{1,\,\frac{\left\|\left|\grad^k f\right|\right\|_{\Leb^q}^{2}}{\Vop^k(f)},\, \frac{\sum_{J:|J|=k}\left\|\Dop_J f\right\|_{\Leb^q}^2}{\Vop^k(f)}\right\}^{\frac{q}{2-q}\tanh(\epsilon)}\\
    &= \upe^{-2 k\Rscr(f,q) \tanh(\epsilon)} \Vop^k(f).
\end{align*}
Rewriting the inequality and recalling the definition of $\Vop^k$ yields
\begin{equation*}
    \left(1-\upe^{-2\Rscr(f,q) \tanh(\epsilon)}\right) \Vop^k(f)^{1/k} \leq \Vop^k(f)^{1/k} - \Vop^k(\Pop_\epsilon f)^{1/k} = \int_{0}^{\epsilon} 2 \left\|\left|\grad^k \Pop_s f\right|\right\|_{\Leb^2}^{2/k} \dif s.
\end{equation*}
On the other hand, by Lemma~\ref{lem:energy-bound-Boolean}, for all $p\in [1,2]$ we have
\begin{equation*}
    \int_{0}^{\epsilon} 2 \left\|\left|\grad^k \Pop_s f\right|\right\|_{\Leb^2}^{2/k} \dif s \leq \left(\left\|f\right\|_{\Leb^{\infty}}^{2-p} \left\|\left|\grad^k f\right|\right\|_{\Leb^p}^{p}\right)^{1/k}\int_{0}^{\epsilon} 2\, \upe^{-2\left(p-1\right)s}\left(\upe^{4s}-1\right)^{-(1-\frac{p}{2})} \dif s.
\end{equation*}
Here the integral yields an incomplete Beta function expression:
\begin{equation*}
    \int_{0}^{\epsilon} 2\, \upe^{-2\left(p-1\right)s}\left(\upe^{4s}-1\right)^{-(1-\frac{p}{2})} \dif s = 2^{p-1}\Beta\left(\frac{1-\upe^{-2\epsilon}}{2};\frac{p}{2},\frac{p}{2}\right).
\end{equation*}
Note that $\frac{1-\upe^{-2\epsilon}}{2} = \frac{\tanh(\epsilon)}{1+\tanh(\epsilon)}$. Lemma~\ref{lem:beta_upper} yields
\begin{equation*}
    \Beta\left(\frac{1-\upe^{-2\epsilon}}{2};\frac{p}{2},\frac{p}{2}\right) \leq \frac{2}{p}\, \tanh(\epsilon)^{\frac{p}{2}}.
\end{equation*}
Rounding up, and substituting $r=2\Rscr(f,q) \tanh(\epsilon)$, we obtain
\begin{equation*}
    \left\|f\right\|_{\Leb^\infty}^{2-p} \left\|\left|\grad^k f\right|\right\|_{\Leb^p}^{p} \geq \Vop^k(f) \left(p \, 2^{-\frac{p}{2}}  \frac{1-\upe^{-r}}{r^{p/2}}\right)^{k} \Rscr(f,q)^{\frac{kp}{2}}.
\end{equation*}
Here $r$ takes value in $(0,2\Rscr(f,q)\frac{2-q}{q}]$. Note that $q\leq p$ and $\Rscr(f,q)\geq 1$. By Lemma~\ref{lem:phi_max_location},
\begin{equation*}
    2\Rscr(f,q)\,\frac{2-q}{q} \geq \frac{2(2-p)}{p} \geq \argmax_{x\geq 0} \frac{1-\upe^{-x}}{x^{p/2}}.
\end{equation*}
Hence optimizing over $r\in (0,2\Rscr(f,q)\frac{2-q}{q}]$ yields
\begin{equation*}
    \left\|f\right\|_{\Leb^\infty}^{2-p} \left\|\left|\grad^k f\right|\right\|_{\Leb^p}^{p} \geq M(p)^k \Vop^k(f)\,\Rscr(f,q)^{\frac{kp}{2}},
\end{equation*}
with $M(p)=p \, 2^{-\frac{p}{2}} \max_{x\geq 0} \frac{1-\upe^{-x}}{x^{p/2}}$ as desired.
\end{proof}

Compared to the global estimate with the factor $M(p)^k$, the operator-boundedness of the discrete derivatives yields a sharper partial high-order energy--variance bound with an optimal constant of $1$ in the Boolean setting:

\begin{lemma}\label{lem:high-order-inf-var-Boolean}
For all $J\subseteq [n]$, $f:\cube^n\to \Rbb$, $p\in [1,2]$ and $q\in [1,2)$,
\begin{equation*}
    \left\|f\right\|_{\Leb^\infty}^{2-p} \left\|\Dop_J f\right\|_{\Leb^p}^{p} \geq \Vop_J(f) \left[1+\frac{q}{2|J|(2-q)} \ln^{+}\left(\frac{\Vop_J(f)}{\left\|\Dop_J f\right\|_{\Leb^q}^{2}}\right)\right]^{|J|}.
\end{equation*}
\end{lemma}
\begin{proof}
Assume w.l.o.g. $|J|\geq 1$. Denote by $\epsilon_0(q):=\arctanh(\frac{2-q}{q})$ and
\begin{equation*}
    \Rscr(f,q) := 1+\frac{q}{2|J|(2-q)} \ln^{+}\left(\frac{\Vop_J(f)}{\left\|\Dop_J f\right\|_{\Leb^q}^{2}}\right).
\end{equation*}
For all $\epsilon\in [0,\epsilon_0(q)]$, by \eqref{eq:partial-var-decay-Boolean} in Lemma~\ref{lem:var-decay-Boolean} and the fact that $\epsilon\geq \tanh(\epsilon)$, we have
\begin{align*}
    \Vop_J(\Pop_\epsilon f) &\leq\upe^{-2\epsilon |J|} \Vop_J(f) \min\left\{1,\,\frac{\left\|\Dop_J f\right\|_{\Leb^q}^2}{\Vop_J(f)}\right\}^{\frac{q}{2-q}\tanh(\epsilon)}\\
    &\leq \upe^{-2|J|\tanh(\epsilon)} \Vop_J(f) \min\left\{1,\,\frac{\left\|\Dop_J f\right\|_{\Leb^q}^2}{\Vop_J(f)}\right\}^{\frac{q}{2-q}\tanh(\epsilon)}\\
    &= \upe^{-2 |J|\Rscr(f,q)\tanh(\epsilon)} \Vop_J(f).
\end{align*}
Rewriting the inequality yields
\begin{equation}\label{eq:J-var-decay-Boolean-rewrite}
    \Vop_J(\Pop_\epsilon f)^{1/|J|} \leq \upe^{-2 \Rscr(f,q)\tanh(\epsilon)} \Vop_J(f)^{1/|J|}.
\end{equation}
Since \eqref{eq:J-var-decay-Boolean-rewrite} holds for all $\epsilon \in [0, \epsilon_0(q)]$ with equality at $\epsilon = 0$, by comparing the right-hand derivatives of both sides at $\epsilon = 0$, we obtain
\begin{equation*}
    \left.\frac{\dif}{\dif \epsilon} \Vop_J(\Pop_\epsilon f)^{1/|J|}\right|_{\epsilon = 0}  \leq \left.\frac{\dif}{\dif \epsilon} \left(\upe^{-2 \Rscr(f,q)\tanh(\epsilon)} \Vop_J(f)^{1/|J|}\right) \right|_{\epsilon = 0} = -2\Vop_J(f)^{1/|J|} \Rscr(f,q).
\end{equation*}
On the other hand, by the definition of $\Vop_J$,
\begin{equation*}
    \left.\frac{\dif}{\dif \epsilon} \Vop_J(\Pop_\epsilon f)^{1/|J|}\right|_{\epsilon = 0} = \left.\frac{\dif}{\dif \epsilon} \left(\int_\epsilon^\infty 2 \left\|\Dop_J \Pop_t f\right\|_{\Leb^2}^{2/|J|} \dif t \right)\right|_{\epsilon = 0} = -2 \left\|\Dop_J f\right\|_{\Leb^2}^{2/|J|}.
\end{equation*}
Further by the operator norm bound $\left\|\Dop_J\right\|_{\Leb^\infty\to\Leb^\infty}=1$ (Lemma~\ref{lem:derivative-operator-norm}),
\begin{equation*}
    \left\|\Dop_J f\right\|_{\Leb^2}^2 = \Eop\left[\left|\Dop_J f\right|^2\right] \leq \left\|\Dop_J f\right\|_{\Leb^\infty}^{2-p} \left\|\Dop_J f\right\|_{\Leb^p}^p \leq  \left\|f\right\|_{\Leb^\infty}^{2-p} \left\|\Dop_J f\right\|_{\Leb^p}^{p}.
\end{equation*}
Rounding up, we obtain
\begin{equation*}
    \left\|f\right\|_{\Leb^\infty}^{2-p} \left\|\Dop_J f\right\|_{\Leb^p}^{p} \geq \left\|\Dop_J f\right\|_{\Leb^2}^2 \geq \Vop_J(f) \,\Rscr(f,q)^{|J|}
\end{equation*}
as desired.
\end{proof}

In the Gaussian setting, utilizing the parallel machinery of Lemmas \ref{lem:energy-bound-Gaussian} and \ref{lem:var-decay-Gaussian} via an optimization argument analogous to that of Lemma~\ref{lem:high-order-energy-var-Boolean} yields both the total and partial high-order energy--variance inequalities.

\begin{lemma}\label{lem:high-order-energy-var-Gaussian}
For all $k\in \Nbb$, $f\in \Sob^k(\Rbb^n,\gamma_n)\cap \Leb^{\infty}(\Rbb^n,\gamma_n)$, and $1\leq q\leq p\leq 2$,
\begin{equation}\label{eq:high-order-energy-var-Gaussian}
\begin{aligned}
    &\left\|f\right\|_{\Leb^\infty}^{2-p} \left\|\left|\grad^k f\right|\right\|_{\Leb^p}^{p} \geq M(p)^k\Vop^{k}(f) \\
    &\cdot\left[1+ \frac{q}{2k(2-q)}\max\left\{\ln^+\left(\frac{\Vop^k(f)}{\left\|\left|\grad^k f\right|\right\|_{\Leb^q}^2}\right),\, \ln^+\left(\frac{\Vop^k(f)}{\sum_{|\kappa|=k} \left\|\Dop^\kappa f\right\|_{\Leb^q}^2}\right)\right\}\right]^\frac{kp}{2},
\end{aligned}
\end{equation}
and further for all $\kappa\in \Nbb^n$ with $|\kappa|=k$,
\begin{equation}\label{eq:high-order-inf-var-Gaussian}
    \left\|f\right\|_{\Leb^\infty}^{2-p} \left\|\Dop^\kappa f\right\|_{\Leb^p}^{p} \geq M(p)^{k} \Vop^\kappa(f) \left[1+\frac{q}{2k(2-q)} \ln^{+}\left(\frac{\Vop^\kappa(f)}{\left\|\Dop^\kappa f\right\|_{\Leb^q}^{2}}\right)\right]^\frac{kp}{2}.
\end{equation}
\end{lemma}

\subsection{High-order Poincar\'{e}-type inequalities}\label{subsec:high-order-Poincare-proof}

Now we present the proofs of Theorems \ref{thm:high-order-Poincare-type-Boolean} and \ref{thm:high-order-Poincare-type-Gaussian}.

\begin{proof}[Proof of Theorem \ref{thm:high-order-Poincare-type-Boolean}.]
We begin by observing that a direct application of Lemma~\ref{lem:high-order-energy-var-Boolean} with $\left\|f\right\|_{\Leb^\infty}\leq 1$ immediately yields
\begin{equation*}
    \left\|\left|\grad^k f\right|\right\|_{\Leb^p}^{p} \geq M(p)^k\Vop^{k}(f).
\end{equation*}
Combining this with the lower bound $\Vop^{k}(f)\geq \Wop^{\geq k}(f)$ from Proposition~\ref{prop:high-order-var-bound-Boolean} delivers the qualitative statement. To extract the optimal constant, we refine the estimate by leveraging the baseline variance-decay estimate from Lemma~\ref{lem:var-decay-Boolean}~\eqref{eq:var-decay-Boolean}:
\begin{equation*}
    \Vop^k(\Pop_\epsilon f) \leq \upe^{-2k\epsilon} \Vop^k(f), \quad \forall\, \epsilon>0.
\end{equation*}
An optimization discussion analogous to the proof of Lemma~\ref{lem:high-order-energy-var-Boolean}, invoking the interpolation bound from Lemma~\ref{lem:energy-bound-Boolean}, yields
\begin{equation*}
    \left\|\left|\grad^k f\right|\right\|_{\Leb^p}^{p} \geq \Vop^k(f) \left(2^{1-p}\sup_{\epsilon>0} \frac{1-\upe^{-2\epsilon}}{\Beta\left(\frac{1-\upe^{-2\epsilon}}{2};\frac{p}{2},\frac{p}{2}\right)}\right)^{k}.
\end{equation*}
Lower-bounding the supremum by its asymptotic limit as $\epsilon\to \infty$ yields
\begin{equation*}
    2^{1-p}\sup_{\epsilon>0} \frac{1-\upe^{-2\epsilon}}{\Beta\left(\frac{1-\upe^{-2\epsilon}}{2};\frac{p}{2},\frac{p}{2}\right)} \geq 2^{1-p}\lim_{\epsilon\to \infty} \frac{1-\upe^{-2\epsilon}}{\Beta\left(\frac{1-\upe^{-2\epsilon}}{2};\frac{p}{2},\frac{p}{2}\right)} = \frac{2^{1-p}}{\Beta\left(\frac{1}{2};\frac{p}{2},\frac{p}{2}\right)}.
\end{equation*}
Finally, factoring in the bound $\Vop^k(f) \geq \Wop^{\geq k}(f)$, we obtain
\begin{equation*}
    \left\|\left|\grad^k f\right|\right\|_{\Leb^p}^{p} \geq C^{\text{\rm\ref{thm:high-order-Poincare-type-Boolean}}}(p)^{k} \Wop^{\geq k}(f),
\end{equation*}
with the desired sharper constant $C^{\text{\rm\ref{thm:high-order-Poincare-type-Boolean}}}(p)= 2^{1-p}/\Beta\left(\frac{1}{2};\frac{p}{2},\frac{p}{2}\right)\geq M(p)$.
\end{proof}

\begin{proof}[Proof of Theorem \ref{thm:high-order-Poincare-type-Gaussian}.]
The proof proceeds mutatis mutandis along the same lines as Theorem~\ref{thm:high-order-Poincare-type-Boolean}. Substituting the interpolation bound \eqref{eq:energy-bound-Gaussian} from Lemma~\ref{lem:energy-bound-Gaussian} and the baseline variance-decay estimate from Lemma~\ref{lem:var-decay-Gaussian}~\eqref{eq:var-decay-Gaussian} into the optimization argument, we recover the identical continuous constant $C^{\text{\rm\ref{thm:high-order-Poincare-type-Gaussian}}}(p) = C^{\text{\rm\ref{thm:high-order-Poincare-type-Boolean}}}(p)$.
\end{proof}

\subsection{High-order $\Leb^p$--$\Leb^q$ inequalities}\label{subsec:high-order-Lp-Lq-proof}

Next, we derive the high-order $\Leb^p$--$\Leb^q$ bounds (Theorems \ref{thm:high-order-Lp-Lq-inf-Boolean}, \ref{thm:high-order-Lp-Lq-energy-Boolean} and \ref{thm:high-order-Lp-Lq-Gaussian}) from the master high-order energy--variance inequalities.

In the Boolean setting, by invoking the partial high-order energy--variance inequality from Lemma~\ref{lem:high-order-inf-var-Boolean}, we deduce the high-order $\Leb^p$--$\Leb^q$ influence inequality stated in Theorem~\ref{thm:high-order-Lp-Lq-inf-Boolean}:
\begin{proof}[Proof of Theorem \ref{thm:high-order-Lp-Lq-inf-Boolean}.]
It suffices to show the partial bound for every $J\subseteq [n]$ with $|J|=k$:
\begin{equation}\label{eq:high-order-partial-Tal-KKL-Boolean}
    \frac{\left\|\Dop_J f\right\|_{\Leb^p}^{p}}{\big[1+\frac{q}{2k(2-q)} \ln^{+}(\left\|\Dop_J f\right\|_{\Leb^p}^{p}/\left\|\Dop_J f\right\|_{\Leb^q}^{2})\big]^{k}} \geq C^{\text{\rm\ref{thm:high-order-Lp-Lq-inf-Boolean}}}(q)^k \Vop_J(f).
\end{equation}
The desired result then follows by summing \eqref{eq:high-order-partial-Tal-KKL-Boolean} over all $J\subseteq [n]$ with $|J|=k$, combined with the bound $\sum_{J:|J|=k} \Vop_J(f) \geq \Wop^{\geq k}(f)$ from Proposition~\ref{prop:high-order-var-bound-Boolean}. Assume w.l.o.g. $k\geq 1$ and $\Vop_J(f)>0$. To establish \eqref{eq:high-order-partial-Tal-KKL-Boolean}, we invoke Lemma~\ref{lem:high-order-inf-var-Boolean} with $\left\|f\right\|_{\Leb^\infty}\leq 1$ to obtain
\begin{equation*}
    \left\|\Dop_J f\right\|_{\Leb^p}^{p} \geq \Vop_J(f) \left[1+\frac{q}{2k(2-q)} \ln^{+}\left(\frac{\Vop_J(f)}{\left\|\Dop_J f\right\|_{\Leb^q}^{2}}\right)\right]^{k}.
\end{equation*}
Rewriting this inequality yields
\begin{equation*}
    \left(\frac{\left\|\Dop_J f\right\|_{\Leb^p}^{p}}{\Vop_J(f)}\right)^{\frac{1}{k}} \geq 1+\frac{q}{2k(2-q)} \ln^{+}\left(\frac{\Vop_J(f)}{\left\|\Dop_J f\right\|_{\Leb^q}^{2}}\right).
\end{equation*}
Utilizing the elementary facts $\upe^{-1} x\geq \ln^+ (x)$ and $\ln^+(x)+\ln^+(y)\geq \ln^+(xy)$ yields
\begin{align*}
    &\left(\frac{\left\|\Dop_J f\right\|_{\Leb^p}^{p}}{\Vop_J(f)}\right)^{\frac{1}{k}} + \frac{q}{2k(2-q)}\, k\, \upe^{-1} \left(\frac{\left\|\Dop_J f\right\|_{\Leb^p}^{p}}{\Vop_J(f)}\right)^{\frac{1}{k}}\\
    &\geq 1+\frac{q}{2k(2-q)} \ln^{+}\left(\frac{\Vop_J(f)}{\left\|\Dop_J f\right\|_{\Leb^q}^{2}}\right) + \frac{q}{2k(2-q)}\, k \ln^{+}\left[\left(\frac{\left\|\Dop_J f\right\|_{\Leb^p}^{p}}{\Vop_J(f)}\right)^{\frac{1}{k}}\right]\\
    &\geq 1+\frac{q}{2k(2-q)} \ln^{+}\left(\frac{\left\|\Dop_J f\right\|_{\Leb^p}^{p}}{\left\|\Dop_J f\right\|_{\Leb^q}^{2}}\right).
\end{align*}
Rearranging terms and raising both sides to the power of $k$, we obtain \eqref{eq:high-order-partial-Tal-KKL-Boolean} with the desired constant $C^{\text{\rm\ref{thm:high-order-Lp-Lq-inf-Boolean}}}(q)=\left(1+\frac{q}{2\upe(2-q)}\right)^{-1}\gtrsim 2-q$, which completes the proof.
\end{proof}

The Boolean high-order $\Leb^p$--$\Leb^q$ energy inequality (Theorem~\ref{thm:high-order-Lp-Lq-energy-Boolean}) follows similarly from the total master bound in Lemma~\ref{lem:high-order-energy-var-Boolean}:

\begin{proof}[Proof of Theorem \ref{thm:high-order-Lp-Lq-energy-Boolean}.]
Applying Lemma~\ref{lem:high-order-energy-var-Boolean} with $\left\|f\right\|_{\Leb^\infty}\leq 1$ yields
\begin{equation*}
    \left\|\left|\grad^k f\right|\right\|_{\Leb^p}^{p} \geq M(p)^k\Vop^{k}(f) \left[1+ \frac{q}{2k(2-q)}\ln^+\left(\frac{\Vop^k(f)}{\left\|\left|\grad^k f\right|\right\|_{\Leb^q}^2}\right)\right]^\frac{kp}{2}.
\end{equation*}
Combining this with the bound $\Vop^{k}(f)\geq \Wop^{\geq k}(f)$ from Proposition~\ref{prop:high-order-var-bound-Boolean}, we obtain
\begin{equation*}
    \left\|\left|\grad^k f\right|\right\|_{\Leb^p}^{p} \geq M(p)^k\Wop^{\geq k}(f) \left[1+ \frac{q}{2k(2-q)}\ln^+\left(\frac{\Wop^{\geq k}(f)}{\left\|\left|\grad^k f\right|\right\|_{\Leb^q}^2}\right)\right]^\frac{kp}{2}.
\end{equation*}
Assuming w.l.o.g. $k\geq 1$ and $\Wop^{\geq k}(f) >0$, we rewrite the expression as
\begin{equation*}
    M(p)^{-\frac{2}{p}}\left(\frac{\left\|\left|\grad^k f\right|\right\|_{\Leb^p}^{p}}{\Wop^{\geq k}(f)}\right)^{\frac{2}{kp}} \geq 1+\frac{q}{2k(2-q)} \ln^{+}\left(\frac{\Wop^{\geq k}(f)}{\left\|\left|\grad^k f\right|\right\|_{\Leb^q}^2}\right).
\end{equation*}
By leveraging the facts $\upe^{-1} x\geq \ln^+ (x)$ and $\ln^+(x)+\ln^+(y)\geq \ln^+(xy)$, we deduce that
\begin{align*}
    &M(p)^{-\frac{2}{p}}\left(\frac{\left\|\left|\grad^k f\right|\right\|_{\Leb^p}^{p}}{\Wop^{\geq k}(f)}\right)^{\frac{2}{kp}} + \frac{q}{2k(2-q)}\, \frac{kp}{2} \, \upe^{-1} \left(\frac{\left\|\left|\grad^k f\right|\right\|_{\Leb^p}^{p}}{\Wop^{\geq k}(f)}\right)^{\frac{2}{kp}}\\
    &\geq 1+\frac{q}{2k(2-q)} \ln^{+}\left(\frac{\Wop^{\geq k}(f)}{\left\|\left|\grad^k f\right|\right\|_{\Leb^q}^2}\right) + \frac{q}{2k(2-q)}\, \frac{kp}{2} \ln^{+}\left[\left(\frac{\left\|\left|\grad^k f\right|\right\|_{\Leb^p}^{p}}{\Wop^{\geq k}(f)}\right)^{\frac{2}{kp}}\right]\\
    &\geq 1+\frac{q}{2k(2-q)} \ln^{+}\left(\frac{\left\|\left|\grad^k f\right|\right\|_{\Leb^p}^{p}}{\left\|\left|\grad^k f\right|\right\|_{\Leb^q}^{2}}\right).
\end{align*}
Rearranging terms and raising both sides to the power of $kp/2$ yields the desired bound
\begin{equation*}
    \frac{\left\|\left|\grad^k f\right|\right\|_{\Leb^p}^{p}}{\big[1+\frac{q}{2k(2-q)} \ln^{+}(\left\|\left|\grad^k f\right|\right\|_{\Leb^p}^{p}/\left\|\left|\grad^k f\right|\right\|_{\Leb^q}^{2})\big]^{kp/2}} \geq C^{\text{\rm\ref{thm:high-order-Lp-Lq-energy-Boolean}}}(p,q)^k \Wop^{\geq k}(f),
\end{equation*}
with constant $C^{\text{\rm\ref{thm:high-order-Lp-Lq-energy-Boolean}}}(p,q) = \left(M(p)^{-\frac{2}{p}}+\frac{pq}{4\upe(2-q)}\right)^{-\frac{p}{2}} \gtrsim 2-q$.
\end{proof}

In the Gaussian setting, by invoking the energy-variance inequalities from Lemma~\ref{lem:high-order-energy-var-Gaussian}, we prove the $\Leb^p$--$\Leb^q$ energy and influence inequalities in Theorem~\ref{thm:high-order-Lp-Lq-Gaussian}:

\begin{proof}[Proof of Theorem \ref{thm:high-order-Lp-Lq-Gaussian}.]
Following the identical algebraic procedure detailed in the proof of Theorem~\ref{thm:high-order-Lp-Lq-energy-Boolean}, we deduce the Gaussian energy inequality \eqref{eq:high-order-Lp-Lq-energy-Gaussian} directly from Lemma~\ref{lem:high-order-energy-var-Gaussian}~\eqref{eq:high-order-energy-var-Gaussian}. The proof of the Gaussian influence inequality \eqref{eq:high-order-Lp-Lq-inf-Gaussian} from Lemma~\ref{lem:high-order-energy-var-Gaussian}~\eqref{eq:high-order-inf-var-Gaussian} follows similarly with minor modifications: we first prove the partial bound for all $\kappa\in \Nbb^n$ with $|\kappa|=k$:
\begin{equation}\label{eq:high-order-partial-Tal-KKL-Gaussian}
    \frac{\left\|\Dop^\kappa f\right\|_{\Leb^p}^{p}}{\big[1+\frac{q}{2k(2-q)} \ln^{+}(\left\|\Dop^\kappa f\right\|_{\Leb^p}^{p}/\left\|\Dop^\kappa f\right\|_{\Leb^q}^{2})\big]^{kp/2}} \geq C^{\text{\rm\ref{thm:high-order-Lp-Lq-Gaussian}}}(p,q)^k \Vop^\kappa(f),
\end{equation}
inequality \eqref{eq:high-order-Lp-Lq-inf-Gaussian} then follows by summing \eqref{eq:high-order-partial-Tal-KKL-Gaussian} over all $\kappa\in \Nbb^n$ with $|\kappa|=k$, combined with the bound $\sum_{\kappa:|\kappa|=k} \Vop^\kappa(f) \geq \Wop^{\geq k}(f)$ from Proposition~\ref{prop:high-order-var-bound-Gaussian}. Assume w.l.o.g. $k\geq 1$ and $\Vop^\kappa(f)>0$.  To establish \eqref{eq:high-order-partial-Tal-KKL-Gaussian}, applying Lemma~\ref{lem:high-order-energy-var-Gaussian}~\eqref{eq:high-order-inf-var-Gaussian} with $\left\|f\right\|_{\Leb^\infty}\leq 1$ yields
\begin{equation*}
    \left\|\Dop^\kappa f\right\|_{\Leb^p}^{p} \geq M(p)^k \Vop^\kappa(f) \left[1+\frac{q}{2k(2-q)} \ln^{+}\left(\frac{\Vop^\kappa(f)}{\left\|\Dop^\kappa f\right\|_{\Leb^q}^{2}}\right)\right]^\frac{kp}{2}.
\end{equation*}
Rewriting this inequality yields
\begin{equation*}
    M(p)^{-\frac{2}{p}}\left(\frac{\left\|\Dop^\kappa f\right\|_{\Leb^p}^{p}}{\Vop^\kappa(f)}\right)^{\frac{2}{kp}} \geq 1+\frac{q}{2k(2-q)} \ln^{+}\left(\frac{\Vop^\kappa(f)}{\left\|\Dop^\kappa f\right\|_{\Leb^q}^{2}}\right).
\end{equation*}
By similarly applying $\upe^{-1} x\geq \ln^+ (x)$ and $\ln^+(x)+\ln^+(y)\geq \ln^+(xy)$, we see that
\begin{align*}
    &M(p)^{-\frac{2}{p}}\left(\frac{\left\|\Dop^\kappa f\right\|_{\Leb^p}^{p}}{\Vop^\kappa(f)}\right)^{\frac{2}{kp}}  + \frac{q}{2k(2-q)}\, \frac{kp}{2} \, \upe^{-1} \left(\frac{\left\|\Dop^\kappa f\right\|_{\Leb^p}^{p}}{\Vop^\kappa(f)}\right)^{\frac{2}{kp}} \\
    &\geq 1+\frac{q}{2k(2-q)} \ln^{+}\left(\frac{\Vop^\kappa(f)}{\left\|\Dop^\kappa f\right\|_{\Leb^q}^{2}}\right) + \frac{q}{2k(2-q)}\, \frac{kp}{2} \ln^{+}\left[\left(\frac{\left\|\Dop^\kappa f\right\|_{\Leb^p}^{p}}{\Vop^\kappa(f)}\right)^{\frac{2}{kp}} \right]\\
    &\geq 1+\frac{q}{2k(2-q)} \ln^{+}\left(\frac{\left\|\Dop^\kappa f\right\|_{\Leb^p}^{p}}{\left\|\Dop^\kappa f\right\|_{\Leb^q}^{2}}\right).
\end{align*}
A similar rearrangement recovers \eqref{eq:high-order-partial-Tal-KKL-Gaussian} with the desired identical constant
\begin{equation*}
    C^{\text{\rm\ref{thm:high-order-Lp-Lq-Gaussian}}}(p,q) = \left(M(p)^{-\frac{2}{p}}+\frac{pq}{4\upe(2-q)}\right)^{-\frac{p}{2}}= C^{\text{\rm\ref{thm:high-order-Lp-Lq-energy-Boolean}}}(p,q),
\end{equation*}
which completes the proof.
\end{proof}

\subsection{High-order isoperimetric-type inequalities}\label{subsec:high-order-iso-proof}

Next, we derive the high-order isoperimetric-type inequalities (Theorems \ref{thm:high-order-Tal-isoper-Boolean}, \ref{thm:high-order-partial-isoper-Boolean} and \ref{thm:high-order-iso-Gaussian}) from the high-order $\Leb^p$--$\Leb^q$ inequalities, followed by the consequent KKL-type bounds.

Using the Boolean energy inequality (Theorem \ref{thm:high-order-Lp-Lq-energy-Boolean}), we first deduce Theorem \ref{thm:high-order-Tal-isoper-Boolean}:
\begin{proof}[Proof of Theorem \ref{thm:high-order-Tal-isoper-Boolean}.]
Setting $q=p$ in Theorem \ref{thm:high-order-Lp-Lq-energy-Boolean} yields
\begin{equation*}
    \frac{\left\|\left|\grad^k f\right|\right\|_{\Leb^p}^{p}}{\big[1+\frac{1}{2k} \ln^{+}(1/\left\|\left|\grad^k f\right|\right\|_{\Leb^p}^{p})\big]^{kp/2}} \geq C^{\text{\rm\ref{thm:high-order-Lp-Lq-energy-Boolean}}}(p,p)^k \Wop^{\geq k}(f).
\end{equation*}
Assuming w.l.o.g. $k\geq 1$ and $\Wop^{\geq k}(f)>0$, we rewrite this inequality as
\begin{equation*}
    C^{\text{\rm\ref{thm:high-order-Lp-Lq-energy-Boolean}}}(p,p)^{-\frac{2}{p}}\left(\frac{\left\|\left|\grad^k f\right|\right\|_{\Leb^p}^{p}}{\Wop^{\geq k}(f)}\right)^{\frac{2}{kp}} \geq 1+ \frac{1}{2k}\ln^+\left(\frac{1}{\left\|\left|\grad^k f\right|\right\|_{\Leb^p}^p}\right).
\end{equation*}
Reapplying $\upe^{-1} x\geq \ln^+ (x)$ and $\ln^+(x)+\ln^+(y)\geq \ln^+(xy)$ further yields
\begin{align*}
    &C^{\text{\rm\ref{thm:high-order-Lp-Lq-energy-Boolean}}}(p,p)^{-\frac{2}{p}}\left(\frac{\left\|\left|\grad^k f\right|\right\|_{\Leb^p}^{p}}{\Wop^{\geq k}(f)}\right)^{\frac{2}{kp}} + \frac{1}{2k}\,\frac{kp}{2}\,\upe^{-1} \left(\frac{\left\|\left|\grad^k f\right|\right\|_{\Leb^p}^{p}}{\Wop^{\geq k}(f)}\right)^{\frac{2}{kp}} \\
    &\geq 1+ \frac{1}{2k}\ln^+\left(\frac{1}{\left\|\left|\grad^k f\right|\right\|_{\Leb^p}^p}\right) + \frac{1}{2k}\,\frac{kp}{2}\ln^+\left[\left(\frac{\left\|\left|\grad^k f\right|\right\|_{\Leb^p}^{p}}{\Wop^{\geq k}(f)}\right)^{\frac{2}{kp}}\right]\\
    &\geq 1 + \frac{1}{2k} \ln^+\left(\frac{1}{\Wop^{\geq k}(f)}\right) = 1 + \frac{1}{2k} \ln\left(\frac{1}{\Wop^{\geq k}(f)}\right).
\end{align*}
Rearranging terms and raising both sides to the power of $\frac{kp}{2}$ yields
\begin{equation*}
    \left\|\left|\grad^k f\right|\right\|_{\Leb^p}^{p} \geq C^{\text{\rm\ref{thm:high-order-Tal-isoper-Boolean}}}(p)^k\Wop^{\geq k}(f) \left[1+ \frac{1}{2k}\ln\left(\frac{1}{\Wop^{\geq k}(f)}\right)\right]^\frac{kp}{2},
\end{equation*}
with the explicit constant
\begin{equation*}
    C^{\text{\rm\ref{thm:high-order-Tal-isoper-Boolean}}}(p) = \left(C^{\text{\rm\ref{thm:high-order-Lp-Lq-energy-Boolean}}}(p,p)^{-\frac{2}{p}}+\frac{p}{4\upe}\right)^{-\frac{p}{2}} = \left(M(p)^{-\frac{2}{p}}+\frac{p}{2\upe(2-p)}\right)^{-\frac{p}{2}} \gtrsim 2-p
\end{equation*}
as desired.
\end{proof}

Theorem \ref{thm:high-order-partial-isoper-Boolean} follows similarly from the Boolean influence inequality (Theorem \ref{thm:high-order-Lp-Lq-inf-Boolean}):
\begin{proof}[Proof of Theorem \ref{thm:high-order-partial-isoper-Boolean}.]
Setting $q=p$ in Theorem \ref{thm:high-order-Lp-Lq-inf-Boolean} yields
\begin{equation*}
    \sum_{J:|J|=k} \frac{\left\|\Dop_J f\right\|_{\Leb^p}^{p}}{\big[1+\frac{1}{2k} \ln^{+}(1/\left\|\Dop_J f\right\|_{\Leb^p}^{p})\big]^{k}} \geq C^{\text{\rm\ref{thm:high-order-Lp-Lq-inf-Boolean}}}(p)^k \Wop^{\geq k}(f).
\end{equation*}
Let $J_{*}:=\argmax_{J:|J|=k} \left\|\Dop_J f\right\|_{\Leb^p}^{p}$. Direct bounding by the maximal influence yields
\begin{equation*}
    \binom{n}{k}\,\frac{\left\|\Dop_{J_{*}} f\right\|_{\Leb^p}^{p}}{\big[1+\frac{1}{2k} \ln^{+}(1/\left\|\Dop_{J_{*}} f\right\|_{\Leb^p}^{p})\big]^{k}}  \geq C^{\text{\rm\ref{thm:high-order-Lp-Lq-inf-Boolean}}}(p)^k \Wop^{\geq k}(f).
\end{equation*}
Assuming w.l.o.g. $k\geq 1$ and $\Wop^{\geq k}(f)>0$, we rewrite this inequality as
\begin{equation*}
    C^{\text{\rm\ref{thm:high-order-Lp-Lq-inf-Boolean}}}(p)^{-1}\left(\frac{\binom{n}{k}\left\|\Dop_{J_*} f\right\|_{\Leb^p}^{2}}{\Wop^{\geq k}(f)}\right)^{\frac{1}{k}} \geq 1+\frac{1}{2k} \ln^{+}\left(\frac{1}{\left\|\Dop_{J_*} f\right\|_{\Leb^p}^{p}}\right).
\end{equation*}
The facts $\upe^{-1} x\geq \ln^+ (x)$ and $\ln^+(x)+\ln^+(y)\geq \ln^+(xy)$ similarly yield
\begin{align*}
    &C^{\text{\rm\ref{thm:high-order-Lp-Lq-inf-Boolean}}}(p)^{-1}\left(\frac{\binom{n}{k}\left\|\Dop_{J_*} f\right\|_{\Leb^p}^{p}}{\Wop^{\geq k}(f)}\right)^{\frac{1}{k}} +\frac{1}{2k}\, k\, \upe^{-1} \left(\frac{\binom{n}{k}\left\|\Dop_{J_*} f\right\|_{\Leb^p}^{p}}{\Wop^{\geq k}(f)}\right)^{\frac{1}{k}}\\
    &\geq 1+\frac{1}{2k} \ln^{+}\left(\frac{1}{\left\|\Dop_{J_*} f\right\|_{\Leb^p}^{p}}\right) + \frac{1}{2k}\, k \ln^{+}\left[\left(\frac{\binom{n}{k}\left\|\Dop_{J_*} f\right\|_{\Leb^p}^{p}}{\Wop^{\geq k}(f)}\right)^{\frac{1}{k}}\right]\\
    &\geq 1+\frac{1}{2k} \ln^{+}\left(\frac{\binom{n}{k}}{\Wop^{\geq k}(f)}\right) = 1+\frac{1}{2k} \ln\left(\frac{\binom{n}{k}}{\Wop^{\geq k}(f)}\right).
\end{align*}
Isolating the maximal influence and raising both sides to the power of $k$ yields
\begin{equation*}
    \max_{J\subseteq[n]:|J|=k}\left\|\Dop_J f\right\|_{\Leb^p}^{p} = \left\|\Dop_{J_*} f\right\|_{\Leb^p}^{p} \geq C^{\text{\rm\ref{thm:high-order-partial-isoper-Boolean}}}(p)^k\, \frac{\Wop^{\geq k}(f)}{\binom{n}{k}} \left[1+ \frac{1}{2k} \ln\left(\frac{\binom{n}{k}}{\Wop^{\geq k}(f)}\right)\right]^k,
\end{equation*}
with the explicit constant
\begin{equation*}
    C^{\text{\rm\ref{thm:high-order-partial-isoper-Boolean}}}(p) = \left(C^{\text{\rm\ref{thm:high-order-Lp-Lq-inf-Boolean}}}(p)^{-1}+\frac{1}{2\upe}\right)^{-1} =  \left(1+\frac{1}{\upe(2-p)}\right)^{-1} \gtrsim 2-p
\end{equation*}
as desired.
\end{proof}

As noted in Section~\ref{subsec:results}, Theorem \ref{thm:high-order-partial-isoper-Boolean} further yields high-order KKL inequalities:

\begin{proof}[Proof of high-order KKL inequalities.]
We first prove the high-order $\Leb^p$ KKL inequality \eqref{eq:high-order-Lp-KKL-Boolean} for bounded functions. Since $\left\|f\right\|_{\Leb^\infty}\leq 1$ implies $\Wop^{\geq k}(f)\leq 1$, Theorem \ref{thm:high-order-partial-isoper-Boolean} provides
\begin{equation*}
    \max_{J\subseteq[n]:|J|=k}\left\|\Dop_J f\right\|_{\Leb^p}^{p} \geq C^{\text{\rm\ref{thm:high-order-partial-isoper-Boolean}}}(p)^k\, \Wop^{\geq k}(f)\frac{\left[1+\frac{1}{2k} \ln\binom{n}{k}\right]^k}{\binom{n}{k}}
\end{equation*}
Assume w.l.o.g. $k\geq 1$. The standard bound $\left(n/k\right)^k\leq \binom{n}{k} \leq \left(\upe n/k\right)^k$ for $n\geq k$ yields
\begin{equation*}
    \frac{\left[1+\frac{1}{2k} \ln\binom{n}{k}\right]^k}{\binom{n}{k}} \geq \frac{\left[1+\frac{1}{2} \ln(n/k)\right]^k}{(\upe n/k)^k} = \left[\frac{k\left(1+\frac{1}{2} \ln(n/k)\right)}{\upe \ln(n)}\,\frac{\ln(n)}{n} \right]^k.
\end{equation*}
Next, observe that
\begin{equation*}
    \left(1+\frac{1}{2} \ln(n/k)\right)\left(2+\ln(k)\right) \geq \ln(n) \iff \frac{\left(1+\frac{1}{2} \ln(n/k)\right)}{\ln(n)} \geq \frac{1}{2+\ln(k)}.
\end{equation*}
Consequently, for every fixed $p\in [1,2)$, it follows that
\begin{equation*}
    C^{\text{\rm\ref{thm:high-order-partial-isoper-Boolean}}}(p)^k \frac{\left[1+\frac{1}{2k} \ln\binom{n}{k}\right]^k}{\binom{n}{k}} \geq \left(\frac{C^{\text{\rm\ref{thm:high-order-partial-isoper-Boolean}}}(p)\,k}{\upe \left(2+\ln(k)\right)}\, \right)^k \left(\frac{\ln(n)}{n}\right)^k \gtrsim_p \left(\frac{\ln(n)}{n}\right)^k,
\end{equation*}
Combining these bounds yields
\begin{equation*}
    \max_{J\subseteq[n]:|J|=k}\left\|\Dop_J f\right\|_{\Leb^p}^{p} \gtrsim_p \Wop^{\geq k}(f) \left(\frac{\ln(n)}{n}\right)^k,
\end{equation*}
which completes the proof of the high-order $\Leb^p$ KKL inequality \eqref{eq:high-order-Lp-KKL-Boolean}.

When restricted to $\cube$-valued $f$, recall that we have $\left\|\Dop_J f\right\|_{\Leb^1} \leq 2^{|J|-1} \Inf_J(f)$. Setting $p=1$ in Theorem \ref{thm:high-order-partial-isoper-Boolean} yields
\begin{equation*}
    \max_{J\subseteq[n]:|J|=k}\Inf_J(f) \geq \left(\frac{C^{\text{\rm\ref{thm:high-order-partial-isoper-Boolean}}}(1)}{2}\right)^k\, \frac{\Wop^{\geq k}(f)}{\binom{n}{k}} \left[1+ \frac{1}{2k} \ln\left(\frac{\binom{n}{k}}{\Wop^{\geq k}(f)}\right)\right]^k.
\end{equation*}
Following an analogous algebraic reduction, we deduce that
\begin{equation*}
    \max_{J\subseteq[n]:|J|=k}\Inf_J(f) \gtrsim \Wop^{\geq k}(f)\left(\frac{\ln(n)}{n}\right)^k,
\end{equation*}
thereby recovering the high-order KKL inequality in \cite{P2025}.
\end{proof}

In the Gaussian setting, by leveraging the $\Leb^p$--$\Leb^q$ energy and influence inequalities in Theorem~\ref{thm:high-order-Lp-Lq-Gaussian}, we similarly prove Theorem~\ref{thm:high-order-iso-Gaussian}:

\begin{proof}[Proof of Theorem \ref{thm:high-order-iso-Gaussian}.]
Following the identical algebraic procedure detailed in the proof of Theorem~\ref{thm:high-order-Tal-isoper-Boolean}, we deduce the Gaussian isoperimetric-type inequality \eqref{eq:high-order-Lp-Lq-energy-Gaussian} directly from Theorem~\ref{thm:high-order-Lp-Lq-Gaussian}~\eqref{eq:high-order-Lp-Lq-energy-Gaussian}. The proof of the Gaussian KKL-type isoperimetric inequality \eqref{eq:high-order-partial-iso-Gaussian} from Theorem~\ref{thm:high-order-Lp-Lq-Gaussian}~\eqref{eq:high-order-Lp-Lq-inf-Gaussian} follows similarly with minor modifications: setting $q=p$ in \eqref{eq:high-order-Lp-Lq-inf-Gaussian} yields
\begin{equation*}
    \sum_{\kappa:|\kappa|=k} \frac{\left\|\Dop^\kappa f\right\|_{\Leb^p}^{p}}{\big[1+\frac{1}{2k} \ln^{+}(1/\left\|\Dop^\kappa f\right\|_{\Leb^p}^{p})\big]^{kp/2}} \geq C^{\text{\rm\ref{thm:high-order-Lp-Lq-Gaussian}}}(p,p)^k \Wop^{\geq k}(f).
\end{equation*}

Using the same discussion as in Theorem \ref{thm:high-order-Tal-isoper-Boolean}, we can deduce \eqref{eq:high-order-iso-Gaussian} from \eqref{eq:high-order-Lp-Lq-energy-Gaussian} in Theorem \ref{thm:high-order-Lp-Lq-Gaussian}. Similarly, applying \eqref{eq:high-order-Lp-Lq-inf-Gaussian} in Theorem \ref{thm:high-order-Lp-Lq-Gaussian} with $q=p$ yields
\begin{equation*}
    \sum_{\kappa:|\kappa|=k} \frac{\left\|\Dop^\kappa f\right\|_{\Leb^p}^{p}}{\big[1+\frac{1}{2k} \ln^{+}(1/\left\|\Dop^\kappa f\right\|_{\Leb^p}^{p})\big]^{kp/2}} \geq C^{\text{\rm\ref{thm:high-order-Lp-Lq-Gaussian}}}(p,p)^k \Wop^{\geq k}(f).
\end{equation*}
Let $\kappa_{*}:=\argmax_{\kappa:|\kappa|=k} \left\|\Dop^\kappa f\right\|_{\Leb^p}^{p}$. Direct bounding by the maximal influence yields
\begin{equation*}
    \binom{n+k-1}{k}\,\frac{\left\|\Dop^{\kappa_*} f\right\|_{\Leb^p}^{p}}{\big[1+\frac{1}{2k} \ln^{+}(1/\left\|\Dop^{\kappa_*} f\right\|_{\Leb^p}^{p})\big]^{kp/2}} \geq C^{\text{\rm\ref{thm:high-order-Lp-Lq-Gaussian}}}(p,p)^k \Wop^{\geq k}(f).
\end{equation*}
Assuming w.l.o.g. $k\geq 1$ and $\Wop^{\geq k}(f)>0$, we rewrite this inequality as
\begin{equation*}
    C^{\text{\rm\ref{thm:high-order-Lp-Lq-Gaussian}}}(p,p)^{-\frac{2}{p}}\left(\frac{\binom{n+k-1}{k}\left\|\Dop^{\kappa_*} f\right\|_{\Leb^p}^{p}}{\Wop^{\geq k}(f)}\right)^{\frac{2}{kp}} \geq 1+\frac{1}{2k} \ln^{+}\left(\frac{1}{\left\|\Dop^{\kappa_*} f\right\|_{\Leb^p}^{2}}\right).
\end{equation*}
Similarly leveraging $\upe^{-1} x\geq \ln^+ (x)$ and $\ln^+(x)+\ln^+(y)\geq \ln^+(xy)$ yields
\begin{align*}
    &C^{\text{\rm\ref{thm:high-order-Lp-Lq-Gaussian}}}(p,p)^{-\frac{2}{p}}\left(\frac{\binom{n+k-1}{k}\left\|\Dop^{\kappa_*} f\right\|_{\Leb^p}^{p}}{\Wop^{\geq k}(f)}\right)^{\frac{2}{kp}}  + \frac{1}{2k}\, \frac{kp}{2}\, \upe^{-1} \left(\frac{\binom{n+k-1}{k}\left\|\Dop^{\kappa_*} f\right\|_{\Leb^p}^{p}}{\Wop^{\geq k}(f)}\right)^{\frac{2}{kp}} \\
    &\geq 1+\frac{1}{2k} \ln^{+}\left(\frac{1}{\left\|\Dop^{\kappa_*} f\right\|_{\Leb^p}^{2}}\right)+ \frac{1}{2k}\, \frac{kp}{2} \ln^{+}\left[\left(\frac{\binom{n+k-1}{k}\left\|\Dop^{\kappa_*} f\right\|_{\Leb^p}^{p}}{\Wop^{\geq k}(f)}\right)^{\frac{2}{kp}}\right]\\
    &\geq 1+\frac{1}{2k} \ln^{+}\left(\frac{\binom{n+k-1}{k}}{\Wop^{\geq k}(f)}\right) = 1+\frac{1}{2k} \ln\left(\frac{\binom{n+k-1}{k}}{\Wop^{\geq k}(f)}\right).
\end{align*}
Isolating the maximal influence and raising both sides to the power of $kp/2$ yields
\begin{equation*}
    \max_{\kappa\in\Nbb^n:|\kappa|=k}\left\|\Dop^\kappa f\right\|_{\Leb^p}^{p} = \left\|\Dop^{\kappa_*} f\right\|_{\Leb^p}^{p} \geq C^{\text{\rm\ref{thm:high-order-iso-Gaussian}}}(p)^k\, \frac{\Wop^{\geq k}(f)}{\binom{n+k-1}{k}} \left[1+ \frac{1}{2k} \ln\left(\frac{\binom{n+k-1}{k}}{\Wop^{\geq k}(f)}\right)\right]^\frac{kp}{2},
\end{equation*}
with the desired identical explicit constant
\begin{equation*}
    C^{\text{\rm\ref{thm:high-order-iso-Gaussian}}}(p)= \left(C^{\text{\rm\ref{thm:high-order-Lp-Lq-Gaussian}}}(p,p)^{-\frac{2}{p}}+\frac{p}{4\upe}\right)^{-\frac{p}{2}} = \left(M(p)^{-\frac{2}{p}}+\frac{p}{2\upe(2-p)}\right)^{-\frac{p}{2}} = C^{\text{\rm\ref{thm:high-order-Tal-isoper-Boolean}}}(p),
\end{equation*}
which completes the proof of \eqref{eq:high-order-partial-iso-Gaussian}.
\end{proof}

Theorem~\ref{thm:high-order-iso-Gaussian}~\eqref{eq:high-order-partial-iso-Gaussian} also yields a Gaussian high-order $\Leb^p$ KKL inequality \eqref{eq:high-order-Lp-KKL-Gaussian}:

\begin{proof}[Proof of the Gaussian high-order $\Leb^p$ KKL inequality \eqref{eq:high-order-Lp-KKL-Gaussian}.]
Theorem~\ref{thm:high-order-iso-Gaussian}~\eqref{eq:high-order-partial-iso-Gaussian} yields
\begin{equation*}
    \max_{\kappa\in\Nbb^n:|\kappa|=k}\left\|\Dop^\kappa f\right\|_{\Leb^p}^{p} \geq C^{\text{\rm\ref{thm:high-order-iso-Gaussian}}}(p)^k\, \Wop^{\geq k}(f) \frac{\left[1+ \frac{1}{2k} \ln \binom{n+k-1}{k}\right]^{kp/2}}{\binom{n+k-1}{k}}.
\end{equation*}
Assume w.l.o.g. $k\geq 1$ and write $m:=n+k-1\geq k$. Similarly invoking the standard bound $\left(m/k\right)^k\leq \binom{m}{k} \leq \left(\upe m/k\right)^k$ and the observation $\frac{\left(1+\frac{1}{2} \ln(m/k)\right)}{\ln(m)} \geq \frac{1}{2+\ln(k)}$, we obtain
\begin{equation*}
    C^{\text{\rm\ref{thm:high-order-iso-Gaussian}}}(p)^k\frac{\left[1+ \frac{1}{2k} \ln \binom{m}{k}\right]^{kp/2}}{\binom{m}{k}} \geq \left[\frac{C^{\text{\rm\ref{thm:high-order-iso-Gaussian}}}(p)\,k}{\upe \left(2+\ln(k)\right)^{p/2}}\, \right]^k \left(\frac{\left(\ln(m)\right)^{\frac{p}{2}}}{m} \right)^k \gtrsim_p \left(\frac{\left(\ln(m)\right)^{\frac{p}{2}}}{m} \right)^k.
\end{equation*}
Combining these bounds yields
\begin{equation*}
    \max_{\kappa\in\Nbb^n:|\kappa|=k}\left\|\Dop^\kappa f\right\|_{\Leb^p}^{p} \gtrsim_p \Wop^{\geq k}(f) \left(\frac{\left(\ln(m)\right)^{\frac{p}{2}}}{m} \right)^k,
\end{equation*}
which completes the proof.
\end{proof}

\subsection{High-order Falik--Samorodnitsky \& Eldan--Gross inequalities}\label{subsec:high-order-FS-EG-proof}

Finally, by combining the master high-order energy--variance inequalities with the high-order isoperimetric-type bounds, we prove the high-order Falik--Samorodnitsky \& Eldan--Gross inequalities (Theorems \ref{thm:high-order-Eldan-Gross-Boolean} and \ref{thm:high-order-Eldan-Gross-Gaussian}).

\begin{proof}[Proof of Theorem \ref{thm:high-order-Eldan-Gross-Boolean}.]
Applying Lemma~\ref{lem:high-order-energy-var-Boolean} with $1\leq q\leq p\leq 2$ in conjunction with the bound $\Vop^k(f)\geq \Wop^{\geq k}(f)$ from Proposition~\ref{prop:high-order-var-bound-Boolean}, we obtain
\begin{equation*}
    \left\|f\right\|_{\Leb^\infty}^{2-p}\left\|\left|\grad^k f\right|\right\|_{\Leb^p}^{p} \geq M(p)^k \Wop^{\geq k}(f) \left[1+ \frac{q}{2k(2-q)}\ln^+\left(\frac{\Wop^{\geq k}(f)}{\sum_{J:|J|=k} \left\|\Dop_J f\right\|_{\Leb^q}^2}\right)\right]^\frac{kp}{2}.
\end{equation*}
Recall that $M(p) = p \, 2^{-\frac{p}{2}} \max_{x\geq 0} \frac{1-\upe^{-x}}{x^{p/2}}$, which satisfies $M(2)=1$. Setting $p=2$, we immediately obtain the high-order Falik--Samorodnitsky inequality \eqref{eq:high-order-Falik-Samorodnitsky-Boolean}:
\begin{equation*}
    \left\|\left|\grad^k f\right|\right\|_{\Leb^2}^{2} \geq \Wop^{\geq k}(f) \left[1+ \frac{q}{2k(2-q)}\ln^+\left(\frac{\Wop^{\geq k}(f)}{\sum_{J:|J|=k} \left\|\Dop_J f\right\|_{\Leb^q}^2}\right)\right]^k.
\end{equation*}
Concurrently, under the assumptions $\left\|f\right\|_{\Leb^\infty}\leq 1$ and $p<2$, we have the isoperimetric-type bound from Theorem~\ref{thm:high-order-Tal-isoper-Boolean}:
\begin{equation*}
    \left\|\left|\grad^k f\right|\right\|_{\Leb^p}^{p} \geq C^{\text{\rm\ref{thm:high-order-Tal-isoper-Boolean}}}(p)^k\Wop^{\geq k}(f) \left[1+ \frac{1}{2k}\ln\left(\frac{1}{\Wop^{\geq k}(f)}\right)\right]^\frac{kp}{2}.
\end{equation*}
Assume w.l.o.g. $\Wop^{\geq k}(f)>0$ and $k\geq 1$. Rewriting these two inequalities and taking a linear combination of them weighted by $1$ and $\frac{q}{2-q}$, respectively, we deduce
\begin{align*}
    &M(p)^{-\frac{2}{p}}\left(\frac{\left\|\left|\grad^k f\right|\right\|_{\Leb^p}^{p}}{\Wop^{\geq k}(f)}\right)^{\frac{2}{kp}} + \frac{q}{2-q} \, C^{\text{\rm\ref{thm:high-order-Tal-isoper-Boolean}}}(p)^{-\frac{2}{p}} \left(\frac{\left\|\left|\grad^k f\right|\right\|_{\Leb^p}^{p}}{\Wop^{\geq k}(f)}\right)^{\frac{2}{kp}}\\
    &\geq 1+ \frac{q}{2k(2-q)}\ln^+\left(\frac{\Wop^{\geq k}(f)}{\sum_{J:|J|=k} \left\|\Dop_J f\right\|_{\Leb^q}^2}\right) + \frac{q}{2k(2-q)} \ln\left(\frac{1}{\Wop^{\geq k}(f)}\right)\\
    &\geq 1+ \frac{q}{2k(2-q)}\ln^+\left(\frac{1}{\sum_{J:|J|=k} \left\|\Dop_J f\right\|_{\Leb^q}^2}\right).
\end{align*}
Rearranging terms and raising both sides to the power of $\frac{kp}{2}$ yields
\begin{equation*}
    \left\|\left|\grad^k f\right|\right\|_{\Leb^p}^{p} \geq C^{\text{\rm\ref{thm:high-order-Eldan-Gross-Boolean}}}(p,q)^k\Wop^{\geq k}(f) \left[1 + \frac{q}{2k(2-q)}\ln^+\left(\frac{1}{\sum_{J:|J|=k} \left\|\Dop_J f\right\|_{\Leb^q}^2} \right)\right]^\frac{kp}{2},
\end{equation*}
with the desired explicit constant
\begin{align*}
    C^{\text{\rm\ref{thm:high-order-Eldan-Gross-Boolean}}}(p,q) &=\left(M(p)^{-\frac{2}{p}} + \frac{q}{2-q}\,C^{\text{\rm\ref{thm:high-order-Tal-isoper-Boolean}}}(p)^{-\frac{2}{p}}\right)^{-\frac{p}{2}} \\
    &= \left(\frac{2}{2-q}\,M(p)^{-\frac{2}{p}}+\frac{pq}{2\upe(2-p)(2-q)}\right)^{-\frac{p}{2}} \gtrsim \left(2-p\right)\left(2-q\right),
\end{align*}
which proves the high-order Eldan--Gross inequality \eqref{eq:high-order-Eldan-Gross-Boolean}.
\end{proof}

\begin{proof}[Proof of Theorem \ref{thm:high-order-Eldan-Gross-Gaussian}.]
The proof proceeds mutatis mutandis along the same lines as the proof of Theorem~\ref{thm:high-order-Eldan-Gross-Boolean}. Substituting the Gaussian total energy--variance inequality \eqref{eq:high-order-iso-Gaussian} from Lemma~\ref{lem:high-order-energy-var-Gaussian} and the isoperimetric-type bound \eqref{eq:high-order-iso-Gaussian} from Theorem~\ref{thm:high-order-iso-Gaussian} into the same algebraic procedure, we obtain the desired Gaussian high-order Falik--Samorodnitsky \& Eldan--Gross inequalities with the identical constant $C^{\text{\rm\ref{thm:high-order-Eldan-Gross-Gaussian}}}(p,q)=C^{\text{\rm\ref{thm:high-order-Eldan-Gross-Boolean}}}(p,q)$.
\end{proof}

\appendix

\section{Useful estimates and analytic facts}

\begin{lemma}\label{lem:phi_max_location}
Let $a\in[1/2,1]$. Then
\begin{equation*}
    \argmax_{x\geq 0} \frac{1-\upe^{-x}}{x^{a}} \leq \frac{2(1-a)}{a}.
\end{equation*}
\end{lemma}

\begin{proof}
Write $\varphi_a(x):=\frac{1-\upe^{-x}}{x^{a}}$. Then for $x>0$,
\begin{equation*}
    \varphi_a^{\prime}(x) = x^{-a-1} \upe^{-x} \left( x- a\left(\upe^x-1\right)\right).
\end{equation*}
\begin{itemize}
    \item When $a=1$, for all $x>0$, note that $\upe^x-1\geq x$, we have
    \begin{equation*}
        \varphi_1^{\prime}(x) = x^{-a-1} \upe^{-x} \left( x- \left(\upe^x-1\right)\right)  \leq 0.
    \end{equation*}
    Hence $\varphi_1(x)$ is decreasing on $(0,\infty)$, and equivalently, maximized at $x=0 = \frac{2(1-a)}{a}$.
    \item When $a\in [1/2,1)$, $\varphi_a(x)$ is maximized by its unique critical point $x_*>0$, which satisfies
    \begin{equation*}
        \varphi_a^{\prime}(x_*)=0 \iff x_* = a\left(\upe^{x_*}-1\right).
    \end{equation*}
    Further using $\upe^{x}-1 \geq x+\frac{1}{2}x^2$ for $x>0$, we have
    \begin{equation*}
        x_* = a\left(\upe^{x_*}-1\right) \geq a\left(x_*+\frac{1}{2}x_*^2\right) \iff 0\leq x_* \leq \frac{2(1-a)}{a}.
    \end{equation*}
\end{itemize}
Hence, in both cases, the maximizer of $\varphi_a$ in $[0,\infty)$ is bounded by $\frac{2(1-a)}{a}$ as desired.
\end{proof}

\begin{lemma}\label{lem:beta_upper}
Let $a\in[1/2,1]$. Then for all $r\in[0,1]$,
\begin{equation*}
    \Beta\left(\frac{r}{1+r};a,a\right)\le \frac{r^a}{a}.
\end{equation*}
\end{lemma}
\begin{proof}
Substituting $t=\frac{u}{1+u}$, we obtain
\begin{equation*}
    \Beta\left(\frac{r}{1+r};a,a\right) = \int_0^{\frac{r}{1+r}} t^{a-1}(1-t)^{a-1}\dif t = \int_0^r u^{a-1} \left(1+u\right)^{-2a}\dif u.
\end{equation*}
Note that for $u\in [0,r]$, $\left(1+u\right)^{-2a}\leq 1$. Hence we obtain
\begin{equation*}
    \Beta\left(\frac{r}{1+r};a,a\right) = \int_0^r u^{a-1} \left(1+u\right)^{-2a}\dif u \leq \int_0^r u^{a-1} \dif u = \frac{r^a}{a}.
\end{equation*}
as desired.
\end{proof}

\begin{lemma}
\label{lem:norm-log-convex}
For a fixed function $f$, the map $r\mapsto \left\|f\right\|_{\Leb^{1/r}}$ is logarithmically convex. More specifically, for $\lambda\in [0,1]$, $p_1, p_2 \in [1,\infty]$ and $\frac{1}{p_\lambda} = \frac{1-\lambda}{p_1} + \frac{\lambda}{p_2}$, we have
\begin{equation*}
    \left\|f\right\|_{\Leb^{p_\lambda}} \leq \left\|f\right\|_{\Leb^{p_1}}^{1-\lambda}  \left\|f\right\|_{\Leb^{p_2}}^{\lambda}.
\end{equation*}
\end{lemma}

\begin{remark}
In particular, Lemma~\ref{lem:norm-log-convex} yields the following useful boundary cases:
\begin{itemize}
    \item For $q \in [2,\infty]$, writing $\frac{1}{q} = \frac{2}{q} \cdot \frac{1}{2} + \frac{q-2}{q} \cdot \frac{1}{\infty}$ implies $\left\|f\right\|_{\Leb^q} \leq \left\|f\right\|_{\Leb^2}^{\frac{2}{q}} \left\|f\right\|_{\Leb^\infty}^{\frac{q-2}{q}}$.
    \item For $p,q\in [1,2)$ and $\vartheta=\min\left\{\frac{1/p-1/2}{1/q-1/2},1\right\}$, we have $\left\|f\right\|_{\Leb^{p}} \leq \left\|f\right\|_{\Leb^{q}}^{\vartheta} \left\|f\right\|_{\Leb^{2}}^{1-\vartheta}$.
\end{itemize}
\end{remark}

\printbibliography[title=References,heading=bibintoc]

\end{document}